\documentclass[11pt,a4paper]{article}
\usepackage[top=1in, bottom=1in, left=1in, right=1in]{geometry}
\usepackage[utf8]{inputenc}
\usepackage[english]{babel}
\usepackage{amssymb}
\usepackage{graphicx}
\usepackage[colorlinks=true, allcolors=blue]{hyperref}
\usepackage{tikz}
\usepackage{amsthm}
\usepackage{amsmath}
\usepackage{cancel}
\usepackage{tikz-cd}
\usepackage{mathtools}
\usepackage{tikz}
\usepackage{pgf}
\usepackage{blindtext}
\usepackage{hyperref}
\usetikzlibrary{matrix,arrows,decorations.pathmorphing}
\tikzset{commutative diagrams/.cd}
\usetikzlibrary{cd}
\usepackage{cancel}
\usepackage{nicefrac}
\usepackage{centernot}
\usepackage{mathtools}
\usepackage{ stmaryrd }
\usepackage{wrapfig}
\usepackage{hyperref}
\hypersetup{
    colorlinks=true,       
    linkcolor=BlueGreen,          
    citecolor=BlueGreen,       
    filecolor=BlueGreen,      
    urlcolor=BlueGreen          
}
\usepackage{soul}
\usepackage[dvipsnames]{xcolor}
\usepackage[sorting=nty,style=alphabetic]{biblatex}
\theoremstyle{plain}% default
\newtheorem{thm}{Theorem}[section]
\newtheorem{lem}[thm]{Lemma}
\newtheorem{prop}[thm]{Proposition}
\newtheorem{cor}[thm]{Corollary}

\newtheorem*{thm*}{Theorem}
\newtheorem*{cor*}{Corollary}

\newtheorem{rcor}{Corollary}
\newtheorem{rthm}{Theorem}

\newenvironment{repcorollary}[1]
  {\renewcommand\thercor{\ref*{#1}}\rcor}
  {\endrcor}
\newenvironment{reptheorem}[1]
  {\renewcommand\therthm{\ref*{#1}}\rthm}
  {\endrthm}

\theoremstyle{definition}
\newtheorem{defn}[thm]{Definition}
\newtheorem*{defn*}{Definition}

\theoremstyle{remark}
\newtheorem{rmk}[thm]{Remark}
\newtheorem*{rmk*}{Remark}
\theoremstyle{remark}

\newtheorem{nt}[thm]{Notation}

\makeatletter
\DeclareFontFamily{OMX}{MnSymbolE}{}
\DeclareSymbolFont{MnLargeSymbols}{OMX}{MnSymbolE}{m}{n}
\SetSymbolFont{MnLargeSymbols}{bold}{OMX}{MnSymbolE}{b}{n}
\DeclareFontShape{OMX}{MnSymbolE}{m}{n}{
    <-6>  MnSymbolE5
   <6-7>  MnSymbolE6
   <7-8>  MnSymbolE7
   <8-9>  MnSymbolE8
   <9-10> MnSymbolE9
  <10-12> MnSymbolE10
  <12->   MnSymbolE12
}{}
\DeclareFontShape{OMX}{MnSymbolE}{b}{n}{
    <-6>  MnSymbolE-Bold5
   <6-7>  MnSymbolE-Bold6
   <7-8>  MnSymbolE-Bold7
   <8-9>  MnSymbolE-Bold8
   <9-10> MnSymbolE-Bold9
  <10-12> MnSymbolE-Bold10
  <12->   MnSymbolE-Bold12
}{}

\let\llangle\@undefined
\let\rrangle\@undefined
\DeclareMathDelimiter{\llangle}{\mathopen}%
                     {MnLargeSymbols}{'164}{MnLargeSymbols}{'164}
\DeclareMathDelimiter{\rrangle}{\mathclose}%
                     {MnLargeSymbols}{'171}{MnLargeSymbols}{'171}
\makeatother

\makeatletter
\newcommand\ackname{Acknowledgements}
\if@titlepage
  \newenvironment{acknowledgements}{%
      \titlepage
      \null\vfil
      \@beginparpenalty\@lowpenalty
      \begin{center}%
        \bfseries \ackname
        \@endparpenalty\@M
      \end{center}}%
     {\par\vfil\null\endtitlepage}
\else
  \newenvironment{acknowledgements}{%
      \if@twocolumn
        \section*{\abstractname}%
      \else
        \small
        \begin{center}%
          {\bfseries \ackname\vspace{-.5em}\vspace{\z@}}%
        \end{center}%
        \quotation
      \fi}
      {\if@twocolumn\else\endquotation\fi}
\fi
\makeatother

\newcommand{\altfrac}[2]{\ifmmode\def\tmp{$}\else\def\tmp{}\fi\mbox{%
    {\raisebox{.24\ht\strutbox}{\tmp#1\tmp}}%
    \kern-2.2pt\scalebox{1.6}[1.5]{/}\kern-1.8pt%
    {\tmp#2\tmp}%
    }}

\newcommand{\frpp}{\mathop{\scalebox{1.5}{\raisebox{-0.2ex}{$\ast$}}}}%

\newcommand{\salbar}{\overline{Sal}(\Gamma)}

\newcommand{\cX}{\mathcal{X}}
\newcommand{\cY}{\mathcal{Y}}

\newcommand{\Sn}{\mathfrak{S}_n}
\newcommand{\A}{\mathrm{A}}
\newcommand{\B}{\mathcal{B}r}
\newcommand{\M}{\mathrm{M}}
\newcommand{\W}{\mathrm{W}}

\newcommand{\VB}{\mathcal{VB}r}
\newcommand{\VA}{\mathrm{VA}}
\newcommand{\KVA}{\mathrm{KVA}}
\newcommand{\PVA}{\mathrm{PVA}}

\newcommand{\CA}{\mathrm{CA}}

\newcommand{\id}{\mathrm{id}}
\newcommand{\hGamma}{\widehat{\Gamma}}
\newcommand{\hm}{\widehat{m}}
\newcommand{\hmbg}{\widehat{m}_{\beta,\gamma}}
\newcommand{\Aut}{\mathrm{Aut}}
\newcommand{\Lm}{\mathrm{L}}
\title{On Decomposability of Virtual Artin Groups}
\author{Federica Gavazzi}
\date{\today}

\begin{document}

\maketitle

\begin{abstract}\noindent
 A group is called decomposable if it can be expressed as a direct product of two proper subgroups, and indecomposable otherwise. This paper explores the decomposability of virtual Artin groups, which were introduced by Bellingeri, Paris, and Thiel as a generalization of classical Artin groups within the framework of virtual braid theory. We establish that for any connected Coxeter graph $\Gamma$, the associated virtual Artin group $\VA[\Gamma]$ is indecomposable. Specifically, virtual braid groups are indecomposable. As a consequence of the indecomposability result, we deduce that studying the automorphism group of a virtual Artin group reduces to analyzing the automorphism groups of its irreducible components.

\medskip
\medskip

{\footnotesize
\noindent \emph{2020 Mathematics Subject Classification.} 20F36.

\noindent \emph{Key words.} Virtual Artin groups, virtual braids, direct decomposition, Artin groups, automorphism groups.}

\end{abstract}

\section*{Introduction}
A group $G$ is said to be \textit{directly decomposable} if it can be expressed as a direct product of two proper subgroups, $G=H\times K$. Otherwise, it is called \textit{directly indecomposable}. In this work, by convention, we will omit the term \textit{directly}. Some groups are well-known examples of indecomposability, such as infinite cyclic groups or free groups. Since the direct factors of a group are necessarily normal, it follows immediately that simple groups are indecomposable. \\
Another notable example is the symmetric group $\Sn$ on $n$ elements, which is indecomposable because its non-trivial normal subgroups must contain the alternating group $\mathrm{Alt_n}$, whose centralizer in $\Sn$ is trivial.\\
\\
\noindent The decomposability problem has been extensively studied for irreducible and finitely generated Coxeter groups. Luis Paris addressed this in \cite{Par07}, as did de Cornulier and de la Harpe in \cite{DeCHarpe}. Later, Nuida extended this investigation to infinitely generated Coxeter groups in \cite{Nui06}.\\
\\
\noindent For Artin groups, Paris demonstrated in \cite{Par04} that irreducible Artin groups of spherical type are indecomposable. However, for affine or more general irreducible Artin groups, the question of indecomposability remains open.\\
\\
\noindent In this work, we investigate the decomposability problem for irreducible virtual Artin groups and establish that they are always indecomposable. As a consequence, studying the automorphism group of these groups reduces to analyzing the automorphism groups of their irreducible components. As part of our analysis, we also prove the indecomposability of certain non-spherical type Artin groups.\\
\\
\noindent Notably, in the study of Artin groups, it is quite rare to obtain results that hold in full generality without imposing assumptions on the Coxeter graphs. Nevertheless, our indecomposability result for virtual Artin groups is one such exception. \\
\\
\noindent Virtual Artin groups were introduced by Bellingeri, Paris and Thiel in \cite{BellParThiel} as a generalization of the concept of virtual braids to all Artin groups. The notion of virtual braids was first defined by Kauffman in his seminal work on virtual knots and links \cite{Kauff}. Virtual braids extend classical braids by introducing additional generators, called virtual generators, which satisfy relations analogous to those of classical braid generators.\\
\\
\noindent The \textit{virtual braid group on $n$ strands}, denoted by $\VB_n$, is defined by a presentation with generators consisting of:
\begin{itemize}
    \item \textit{Classical generators} $\{\sigma_1, \ldots, \sigma_{n-1}\}$, and
    \item \textit{Virtual generators} $\{\tau_1,\ldots ,\tau_{n-1}\}$.
\end{itemize}
\noindent These generators satisfy the following relations:
    \begin{enumerate}
        \item [(VB1)] $\sigma_i\sigma_{i+1}\sigma_i=\sigma_{i+1}\sigma_i\sigma_{i+1}$ for all $1\leq i\leq n-2$; and $\sigma_i\sigma_j=\sigma_j\sigma_i$ for $|i-j|\geq 2$, with $1\leq i,j \leq n-1$;
        \item [(VB2)] $\tau_i\tau_{i+1}\tau_i=\tau_{i+1}\tau_i\tau_{i+1}$ for all $1\leq i\leq n-2$; and $\tau_i\tau_j=\tau_j\tau_i$ for $|i-j|\geq 2$, with $1\leq i,j \leq n-1$; and $\tau_i^2=\id$ for all $1\leq i\leq n-1$;
        \item [(VB3)] $\tau_i\tau_{i+1}\sigma_i=\sigma_{i+1}\tau_i\tau_{i+1}$ for all $1\leq i\leq n-2$; and $\sigma_i\tau_j=\tau_j\sigma_i$ for $|i-j|\geq 2$, with $1\leq i,j \leq n-1$.
    \end{enumerate}
    
\noindent 
The classical generators of $\VB_n$ satisfy the same relations as the generators of the classical braid group $\B_n$, while the virtual generators mimic the behavior of the generators of the symmetric group $\Sn$. The family (VB3) of relations, called \textit{mixed relations}, involves interactions between the classical and virtual generators. \\
\\
\noindent
This concept naturally generalizes to all Artin groups. Recall that given a countable set $S$, a \textit{Coxeter matrix} is a symmetric, square matrix $\M=(m_{s,t})_{s,t\in S}$ indexed by $S$, where $m_{s,t}\in \mathbb{N}_{\geq 2}\cup \{\infty\}$ for all $s,t\in S$, and $m_{s,s}=1$ for all $s\in S$. This data is encoded by a labeled simplicial graph $\Gamma$ called \textit{Coxeter graph}. Its vertices $V(\Gamma)$ correspond to elements of $S$, and an edge between $s$ and $t$ exists if $m_{s,t}\geq 3$. By convention, such an edge is labeled by $m_{s,t}$ if $m_{s,t}\geq 4$. This data will allow us to define the groups that we study in this work.\medskip\\

\noindent The \textit{Artin group} $\A[\Gamma]$ is presented by generators $\mathcal{S}=\{\sigma_s\,|\,s\in S\}$ and relations of the form $\cdots \sigma_s\sigma_t\sigma_s=\cdots \sigma_t\sigma_s\sigma_t$, where both sides have $m_{s,t}$ letters if $m_{s,t}\neq \infty$. No relation is imposed if $m_{s,t}$ is $\infty$. Notably, when $m_{s,t}=2$, the Artin relation reduces to commutation relation.\\
\\
\noindent The \textit{Coxeter group} $\W[\Gamma]$ associated with the Coxeter graph $\Gamma$ is the group presented by generators $s\in S$, and relations of the form $(st)^{m_{s,t}}=\id$ for all $s\neq t$ with $m_{st}\neq \infty$, and $s^2=\id$ for all $s\in S$. The Coxeter group can be viewed as quotient of $\A[\Gamma]$ by the involution relations $\{\sigma_s^2=\id\}$ for all $s\in S$. When the Coxeter graph $\Gamma$ is connected, the groups $\W[\Gamma]$ and $\A[\Gamma]$ are called \textit{irreducible}. If $\Gamma$ decomposes into connected components $\Gamma_1,\ldots, \Gamma_k$, then $\W[\Gamma]=\W[\Gamma_1]\times \cdots \times \W[\Gamma_k]$, and $\A[\Gamma]=\A[\Gamma_1]\times \cdots\times \A[\Gamma_k]$. For this reason, the study of decomposability for Coxeter, Artin, or virtual Artin groups (that we will define in the following) focuses on the irreducible case, as reducible groups decompose naturally into directs products of their irreducible components.\\
\\
\noindent Many standard results on Coxeter and Artin groups are traditionally stated for groups with a finite set of generators $S$. Basic references for Coxeter groups include \cite{Hump} and \cite{Bourbaki}, while for Artin groups, one can consult the lecture notes or the survey paper by Paris (\cite{Par14Lect},\cite{Par14}).\\
\\
\noindent Often, even when working with finitely generated Artin or Coxeter groups, we will consider associated classes of Artin groups (which we denote by $\A[\hGamma]$) that are not necessarily finitely generated. For this reason, in this work we directly consider the most general case where 
$S$ is countable. Most of the known results naturally extend to this broader setting, and when exceptions arise, they will be explicitly noted."\\
\\
\noindent When $S$ is finite and $\W[\Gamma]$ is a finite group, we say that $\Gamma$, $\W[\Gamma]$ and $\A[\Gamma]$ are \textit{of spherical type}. Irreducible Coxeter graphs of spherical type have been completely classified by Coxeter (see \cite{Cox34},\cite{Cox35}), yielding a finite list of possibilities (see Figure \ref{figcoxsph}). \\
\\
\noindent A common example is when $\Gamma=A_{n-1}$ (see Figure \ref{figcoxsph}), where the associated Coxeter group is $\Sn$, and the corresponding Artin group is the braid group on $n$ strands $\mathcal{B}r_n$ .
\\
\\
\noindent To every Coxeter graph $\Gamma$, we can associate a group $\VA[\Gamma]$, which coincides with the virtual braid group on $n$ strands when $\Gamma=A_{n-1}$, as defined above.

\begin{defn*}[\cite{BellParThiel}]
  Let $\Gamma$ be a Coxeter graph with a countable set of vertices $S$. Let $\mathcal{S}=\{\sigma_s\,|\,s\in S\}$ and $\mathcal{T}=\{\tau_s\,|\, s\in S\}$ be two abstract sets in bijection with $S$. The \textit{virtual Artin group} $\VA[\Gamma]$ is defined by a presentation with generators $\mathcal{S}\cup \mathcal{T}$, subject to the following relations:
  \begin{itemize}
      \item [(VA1)] $\underbrace{\cdots \,\sigma_s \sigma_t\sigma_s}_{m_{s,t}\mbox{ \small{letters}}}\;=\;\underbrace{\cdots \,\sigma_t \sigma_s\sigma_t}_{m_{s,t}\mbox{ \small{letters}}}$, for all $s\neq t$, $s,t\in S$ and $m_{s,t}\neq \infty$;
      \item [(VA2)] $\underbrace{\cdots \,\tau_s \tau_t\tau_s}_{m_{s,t}\mbox{ \small{letters}}}\;=\;\underbrace{\cdots \,\tau_t \tau_s\tau_t}_{m_{s,t}\mbox{ \small{letters}}}$, for all $s\neq t$, $s,t\in S$ and $m_{s,t}\neq \infty$; and $\tau_s^2=\id$ for all $s\in S$;
      \item [(VA3)] $\underbrace{\cdots\; \tau_s \tau_t\tau_s}_{m_{s,t}-1\mbox{ \small{letters}}}\sigma_t\,=\,\sigma_r\underbrace{\cdots\; \tau_s \tau_t\tau_s}_{m_{s,t}-1\mbox{\small{ letters}}}$, where $r=s$ if $m_{s,t}$ is even and $r=t$ otherwise, for all $s\neq t$, $s,t\in S$ and $m_{s,t}\neq \infty$.
  \end{itemize}
  
\end{defn*}
\noindent As in the case of virtual braids, the first family of relations (VA1) ensures that the classical generators $\sigma_s$ satisfy the same relations as the generators of the Artin group $\A[\Gamma]$. Similarly, the relations in (VA2) state that the virtual generators $\tau_s$ satisfy the same relations as the generators of the Coxeter group $\W[\Gamma]$. The family (VA3) of mixed relations describes the action of the Coxeter group $\W[\Gamma]$ on its root system (for further details, see \cite{BellParThiel}).
\begin{rmk*}
    If $\Gamma$ is the dihedral Coxeter graph with two vertices $s,t$ and edge labeled by $m$, observe that $(\VA3)$ gives only one relation if $m$ is odd, while it gives two distinct relations if $m$ is even. Indeed, in the first case we have $(\cdots\; \tau_s \tau_t\tau_s)\,\sigma_t\,=\,\sigma_s\,(\cdots\; \tau_s \tau_t\tau_s)$, while in the second case we have $(\cdots\; \tau_s \tau_t\tau_s)\,\sigma_t\,=\,\sigma_t\,(\cdots\; \tau_s \tau_t\tau_s)$ and $(\cdots\; \tau_t\tau_s \tau_t)\,\sigma_s\,=\,\sigma_s\,(\cdots\; \tau_t\tau_s \tau_t)$.
\end{rmk*}

\noindent As with the Coxeter and Artin groups, if $\Gamma_1,\ldots ,\Gamma_k$ are the connected components of $\Gamma$, then $\VA[\Gamma]$ decomposes as $\VA[\Gamma_1]\times \cdots \times \VA[\Gamma_k]$. To study the decomposability of $\VA[\Gamma]$, we focus on its irreducible factors. \textbf{Unless otherwise specified, throughout this article, we will assume $\Gamma$ is connected, and consequently, that the Coxeter, Artin, and virtual Artin groups under consideration are irreducible.}\medskip \\
\noindent This study aims to investigate the decomposability of the virtual Artin group $\VA[\Gamma]$. The main theorem addressing this problem is as follows:

\begin{reptheorem}{vaindec} Let $\Gamma$ be a connected Coxeter graph. Then $\VA[\Gamma]$ is indecomposable. 
\end{reptheorem}

\noindent This theorem is proven in two main steps, distinguishing the case in which $Z(\W[\Gamma])=\{\id\}$ (Lemma \ref{vaindectrivcent}), and the case $Z(\W[\Gamma])\neq \{\id\}$ (Lemma \ref{vaindecnontrivcenter}). These results fully resolve the decomposability problem for irreducible virtual Artin groups. Furthermore, they yield significant insights into the structure of the automorphism group of $\VA[\Gamma]$, that we treat in Section \ref{autgroup}. The relevant theorems are as follows:
\begin{reptheorem}{RemakVA}
Let $\Gamma$ be a Coxeter graph, and let $\Gamma_1,\ldots , \Gamma_k$ be its connected components. Let \\$\VA[\Gamma]=H_1\times \cdots \times H_l$ be a decomposition of the virtual Artin group associated with $\Gamma$ into indecomposable and non-trivial factors. Then, $k=l$ and, up to permutation, $H_i=\VA[\Gamma_i]$ for all $i\in \{1,\ldots,k\}$.
\end{reptheorem}
\begin{reptheorem}{homoautvasn}
Let $\Gamma$ be a Coxeter graph, and let $\Gamma_1,\ldots,\Gamma_k$ be its connected components. Then there exists a homomorphism $\Psi:\Aut(\VA[\Gamma])\longrightarrow \mathfrak{S}_k$ such that, for every $\varphi\in \Aut(\VA[\Gamma])$ and every $i\in \{1,\ldots,k\}$, 
 \[
\varphi(\VA[\Gamma_i])= \VA[\Gamma_{\Psi(\varphi)(i)}].
 \]
\end{reptheorem}
\begin{repcorollary}{autfiniteindex}
Let $\Gamma$ be a Coxeter graph, and let $\Gamma_1,\ldots,\Gamma_k$ be its connected components. Then \\$\Aut(\VA[\Gamma_1])\times \cdots \times \Aut(\VA[\Gamma_k])$ is a finite index normal subgroup of $\Aut(\VA[\Gamma])$.
\end{repcorollary}
\noindent In particular, Corollary \ref{autfiniteindex} implies that to understand the automorphism group of $\VA[\Gamma]$, it suffices to study the automorphism groups of its irreducible components.
\\
\\
\noindent To establish the main results mentioned above, the article is organized as follows. In Section \ref{preliminaries} we summarize the main theorems in \cite{BellParThiel} concerning virtual Artin groups, and in particular we include the definitions of their important subgroups $\PVA[\Gamma]$ and $\KVA[\Gamma]$. The latter is notably an Artin group, relative to a Coxeter graph that the authors denote by $\hGamma$. We also give some preliminaries on the the \textit{colored Artin group} $\CA[\Gamma]$, which is the kernel of the homomorphism $\omega:\A[\Gamma
]\longrightarrow \W[\Gamma]$ mapping $\sigma_s\longmapsto s$ for all $s\in S$. This will be essential in Section \ref{inftyconnindec}.\\
\\
\noindent Then, we examine the decomposability of a certain class of Artin groups associated with some Coxeter graphs that we call \textit{$\infty$-connected} and that we introduce in Section \ref{inftyconnindec}. Intuitively, a Coxeter graph $\Gamma$ is $\infty$-connected if any two vertices can be joined via a path consisting of only $\infty$-labeled edges. We distinguish between the finite and infinite cases and show in Theorem \ref{SfinAindec} that, if $\Gamma$ is a finite $\infty$-connected Coxeter graph, then the Artin group $\A[\Gamma]$ is indecomposable.\\
\\
\noindent The argument to prove Theorem \ref{SfinAindec} can be adapted to the infinite case, relying on the role of the colored Artin group $\CA[\Gamma]$, which is shown to be indecomposable when $S$ is finite and $\Gamma$ is $\infty$-connected (Lemma \ref{SfinCAindecomposable}). 
Additionally, for finite $S$, the centralizer of $\CA[\Gamma]$ in $\A[\Gamma]$ is trivial (Lemma \ref{centerCA}). By extending a result from Godelle and Paris (\cite[Theorem 2.2]{GodPar12}), we prove the indecomposability of $\CA[\Gamma]$ in a specific infinite case (Lemma \ref{SinfCAindec}) and consequently for $\A[\Gamma]$ (Theorem \ref{SinfAindec}).\\
\\
\noindent In Section \ref{kvaindec}, we show that the graph $\hGamma$ (such that $\KVA[\Gamma]\cong \A[\hGamma]$) satisfies the $\infty$-connection property in the finite case, and if infinite, can be expressed as a direct limit of finite $\infty$-connected Coxeter graphs. Thus, $\KVA[\Gamma]=\A[\hGamma]$ always satisfies the hypotheses of Theorem \ref{SinfAindec}, and is therefore indecomposable. \\
\\
\noindent In Section \ref{sectioncentralizerofKVA}, we show that the centralizer of $\KVA[\Gamma]$ in $\VA[\Gamma]$ is trivial. This result is a crucial link between the decomposability of $\KVA[\Gamma]$ and that of the entire group $\VA[\Gamma]$.\\
\\
\noindent In Section \ref{indecva}, we address the indecomposability of $\VA[\Gamma]$. In Subsection \ref{trivialcentercase}, we analyze the case in which $Z(\W[\Gamma])$ is trivial. In Corollary \ref{trivialcenterthenindecompo} we summarize results by Nuida and Paris (\cite{Par07},\cite{Nui06}) showing that the assumption $Z(\W[\Gamma])=\{\id\}$ implies the indecomposability of the Coxeter group $\W[\Gamma]$. This allows to deduce the same property for $\VA[\Gamma]$ in Lemma \ref{vaindectrivcent}. In Subsection \ref{nontrivialcentercase}, we study the case  $Z(\W[\Gamma])\neq\{\id\}$. Here, $\W[\Gamma]$ may admit a decomposition involving its center as a factor. Nevertheless, through some preliminary results on the longest element $w_0$ acting on the root system of $\W[\Gamma]$, we demonstrate that $\VA[\Gamma]$ remains indecomposable.\\
A straightforward consequence of Theorem \ref{vaindec} is the following:

\begin{cor*}
    The virtual braid group on $n$ strands $\VB_n$ is indecomposable.
\end{cor*}

\noindent Finally, in Section \ref{autgroup}, we examine the implications of the indecomposability result for the automorphism group of $\VA[\Gamma]$. Apart from the case of virtual braid groups $\VB_n$ with $n\geq 5$, as studied in \cite{BellPar20}, the automorphism group of virtual Artin groups remains unexplored.

\begin{acknowledgements}\noindent
The author thanks her Ph.D. advisor, Prof. Luis Paris, for his valuable insights and guidance throughout their many discussions. She is also grateful to the anonymous referee of \textit{Journal of Algebra}
for their helpful comments. In addition, she wishes to thank the rapporteurs
of her Ph.D. thesis (of which this article forms a part), Professors Thomas Gobet and Eddy
Godelle, for their careful reading and insightful suggestions.
\end{acknowledgements}

\section{Preliminaries}\label{preliminaries}

\noindent 
In this section we quickly summarize the main prerequisites and results which will be needed in the following analysis. 
\subsection{Virtual Artin groups}\label{subs-virtart}
All relevant information regarding virtual Artin groups can be found in  \cite{BellParThiel}. \medskip\\
\noindent Notably, we emphasize that the Coxeter graphs considered in our analysis have a vertex set $S$ that is not necessarily finite. Let $\mathcal{P}_{\mathrm{fin}}(S)$ denote the set $\{X\subset S\,|\,|X|<\infty\}$. If $X,Y\in \mathcal{P}_{\mathrm{fin}}(S)$ are such that $X\subset Y$, then, by Van der Lek \cite[Lemma 4.11]{van1983homotopy} and Bourbaki \cite[Corollaire 1, Page 19]{Bourbaki} the natural homomorphisms between the respective Artin groups and Coxeter groups 
\begin{align*}
    \A[\Gamma_X]\longrightarrow\A[\Gamma_Y], \qquad\text{and}\qquad \W[\Gamma_X]\longrightarrow\W[\Gamma_Y],
\end{align*}
induced by inclusion are injective. It follows that when $S$ is infinite, the Artin group $\A[\Gamma]$ and the Coxeter group $\W[\Gamma]$ can be expressed as the following direct limits:
\begin{align*}
   \A[\Gamma]\;\;= \underset{X\in \mathcal{P}_{\mathrm{fin}}(S)}{\varinjlim}\A[\Gamma_X]; \qquad \text{ and }\qquad \W[\Gamma]\;\;= \underset{X\in \mathcal{P}_{\mathrm{fin}}(S)}{\varinjlim}\W[\Gamma_X].
\end{align*}

\noindent We saw that the virtual Artin group $\VA[\Gamma]$ associated with a Coxeter graph $\Gamma$ has two distinct sets of generators. The set of classical generators $\mathcal{S}=\{\sigma_s\,|\,s\in S\}$ must satisfy the same relations as the generators of the Artin group $\A[\Gamma]$, while the virtual generators $\mathcal{T}=\{\tau_s\,|\,s\in S\}$ satisfy the same relations as the generators of the Coxeter group $\W[\Gamma]$. A direct consequence of this structure is the existence of two natural group homomorphisms: \begin{align*}
    \iota_{\A}: \A[\Gamma]&\longrightarrow \VA[\Gamma],  &\iota_{\W}:\;\W[\Gamma]&\longrightarrow \VA[\Gamma],\\
    \sigma_s&\longmapsto \sigma_s,  &s &\longmapsto \tau_s.
\end{align*}
In \cite{BellParThiel}, it is shown that these homomorphisms are injective. Consequently, we can view both the Artin group $\A[\Gamma]$ and the Coxeter group $\W[\Gamma]$ as subgroups of the virtual Artin group $\VA[\Gamma]$.
\medskip\\
\noindent
There are two additional group homomorphisms already known for virtual braids: $\pi_K: \VA[\Gamma]\longrightarrow \W[\Gamma]$, which maps $\tau_s$ to $s$ and $\sigma_s$ to $\id$ for all $s\in S$; and $\pi_P: \VA[\Gamma]\longrightarrow \W[\Gamma]$, which maps both $\sigma_s$ and $\tau_s$ to the generator $s$, for all $s$ in $S$.\medskip\\
\noindent It straightforward to verify that the homomorphism $\iota_{\W}$ is a section of both $\pi_K$ and $\pi_P$. Defining the kernels as $\KVA[\Gamma]:=\ker(\pi_K)$ and $\PVA[\Gamma]:=\ker(\pi_P)$, we can deduce that the virtual Artin group $\VA[\Gamma]$ has the following semidirect product structures:
\[\VA[\Gamma]=\KVA[\Gamma]\rtimes \W[\Gamma], \qquad\qquad \VA[\Gamma]=\PVA[\Gamma]\rtimes \W[\Gamma].\]
\begin{nt}\label{notationactionWonKVA}
    Let $\Gamma$ be a Coxeter graph, and let $G$ be either $\PVA[\Gamma]$ or $\KVA[\Gamma]$. For any $w\in \W[\Gamma]$ and $g\in G$, we write $w(g)$ to denote the action of $w$ on $g$ by conjugation, specifically $w(g)=\iota_{\W}(w)\,g\,\iota_{\W}(w)^{-1}$. Sometimes, with slight abuse of notation, we write directly $w$ to mean its image $\iota_{\W}(w)$ in $\VA[\Gamma]$.
\end{nt}
\noindent When $\Gamma=A_{n-1}$, the groups $\PVA[\Gamma]$ and $\KVA[\Gamma]$ were already known as the \textit{pure virtual braid group on $n$ strands}, and the \textit{kure virtual braid group on $n$ strands}, respectively. Group presentations for these were given in \cite{Bard04}, \cite{Rabenda} and \cite{BardBell}.\medskip
\\
\noindent In \cite{BellParThiel}, the authors provide a general group presentation of $\PVA[\Gamma]$ and $\KVA[\Gamma]$, for any Coxeter graph $\Gamma$. Notably, $\KVA[\Gamma]$ is always an Artin group associated with a Coxeter graph $\hGamma$, constructed from $\Gamma$. Both $\PVA[\Gamma]$ and $\KVA[\Gamma]$ have their generators in bijection with the root system of the Coxeter group $\W[\Gamma]$.\medskip
\\
\noindent To clarify, given a Coxeter graph $\Gamma$ on a countable set $S$, we define a formal vector space $V:=\bigoplus_{s\in S} \mathbb{R}\cdot \alpha_s$, where $\Pi=\{\alpha_s\,|\, s\in S\}$ is a set of vectors in bijection with $S$. We define a bilinear symmetric form $\langle\cdot,\cdot\rangle:V\times V\longrightarrow \mathbb{R}$ extending by bilinearity the following rule:
\[
\langle\alpha_s,\alpha_t\rangle:=\begin{cases}
    -\cos{\left(\frac{\pi}{m_{s,t}}\right)}\quad &\mbox{if $m_{s,t}\neq \infty$;}\\
    -1 \quad &\mbox{otherwise.}
\end{cases}
\]
\noindent The \textit{canonical linear representation} of $\W[\Gamma]$ is a group homomorphism $\rho:\W[\Gamma]\longrightarrow GL(V)$, which sends the generator $s$ to the linear transformation defined by: $v\longmapsto v-2\langle v,\alpha_s\rangle\alpha_s $. This representation is known to be faithful (see for instance \cite{Hump}). From now on, the image of a vector $v\in V$ under the action of $w\in \W[\Gamma]$ through $\rho$ will be denoted by $w(v)$. The set $\Phi[\Gamma]=\{w(\alpha_s)\,|\, s\in S;\,w\in \W[\Gamma]\}$ is called the \textit{root system} of $\W[\Gamma]$. This set has been extensively studied (see \cite{deodh82}), and it is worth noting that $\Phi[\Gamma]$ is finite if and only if $\Gamma$ is of spherical type. Irreducible spherical type Coxeter graphs are illustrated in Figure \ref{figcoxsph}.
\begin{figure}
    \centering \includegraphics[width=15cm]{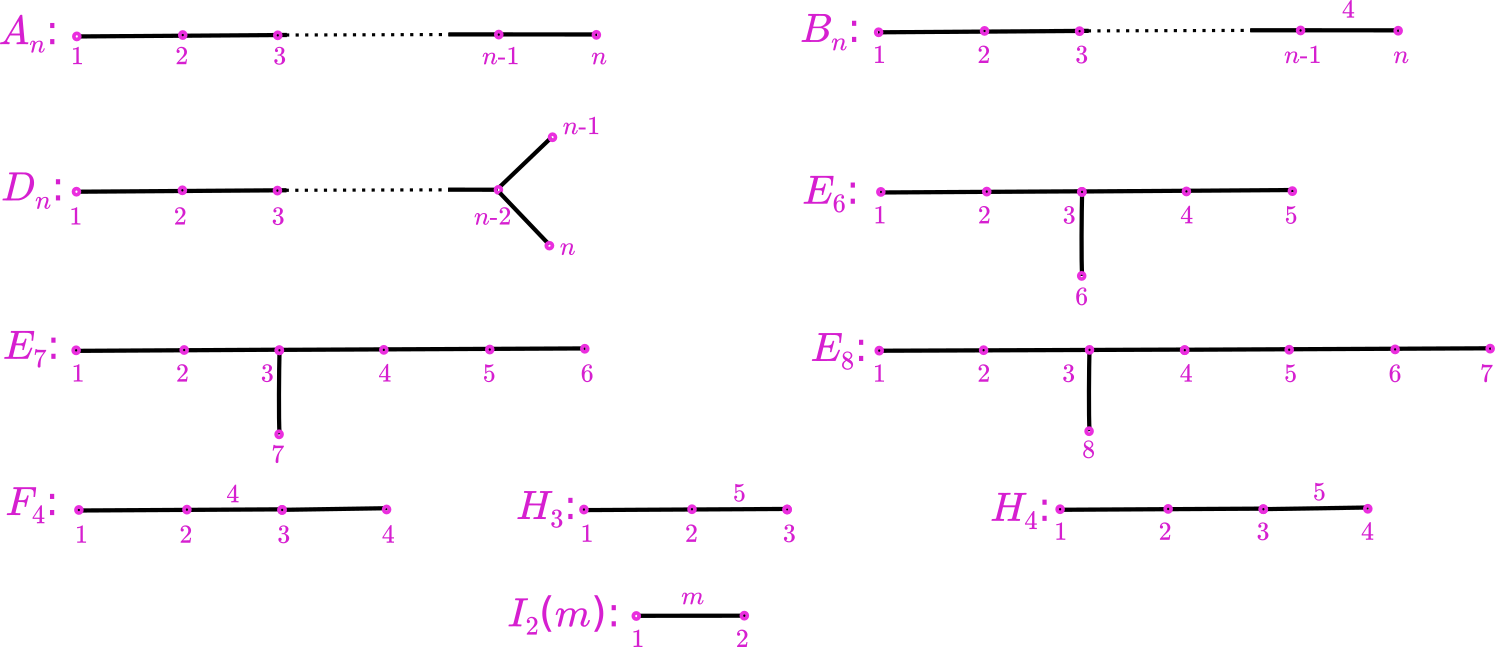}
    \caption{Connected Coxeter graphs of spherical type.}
    \label{figcoxsph}
\end{figure}
\noindent A key result characterizing these groups is that the bilinear symmetric form $\langle\cdot ,\cdot\rangle$ is a scalar product if and only if $\Gamma$ is of spherical type.
\medskip\\
\noindent 
Now, let $\Gamma$ be any Coxeter graph with $\Phi[\Gamma]$ as its associated root system. As described in Notation \ref{notationactionWonKVA},  $w\in \W[\Gamma]$ acts on $\PVA[\Gamma]$ and $\KVA[\Gamma]$. Consider a root $\beta$ in $\Phi[\Gamma]$, which can be expressed as $\beta=w(\alpha_s)$, for some $s\in S$ and $w\in \W[\Gamma]$. Define $\delta_{\beta}:=w\, (\sigma_s)$ and $\zeta_{\beta}=w\,(\tau_s\sigma_s)$, and observe that $\delta_{\beta}\in \KVA[\Gamma]$ and $\zeta_{\beta}\in \PVA[\Gamma]$. By \cite{BellPar20}, the following properties hold:
\begin{itemize}
    \item The definitions of $\delta_{\beta}$ and $\zeta_{\beta}$ are independent of the specific choice of $w\in \W[\Gamma]$ and $s\in S$ \cite[Lemma 2.2]{BellParThiel};
    \item The set $\{\delta_{\beta}\,|\, \beta\in \Phi[\Gamma]\}$ generates $\KVA[\Gamma]$ \cite[Theorem 2.3]{BellParThiel}; 
    \item The set $\{\zeta_{\beta}\;|\;\beta\in \Phi[\Gamma]\}$ generates $\PVA[\Gamma]$ \cite[Theorem 2.6]{BellParThiel};
    \item The action of $\W[\Gamma]$ on the generating sets of $\PVA[\Gamma]$ and $\KVA[\Gamma]$ mirrors its action on the root system $\Phi[\Gamma]$. Specifically, for all $\beta\in \Phi[\Gamma]$ and $w\in \W[\Gamma]$, 
    \[
    w(\delta_{\beta})=\delta_{w(\beta)}, \qquad \qquad \qquad w(\zeta_{\beta})=\zeta_{w(\beta)}.
    \]
    \noindent Remark that $w(g_{\beta})$ for $g$ either $\delta$ or $\zeta$ denotes the action of the Coxeter group $\W[\Gamma]$ on the kernel $G\in \{\PVA[\Gamma],\KVA[\Gamma]\}$ by conjugation, while $w(\beta)$ for $\beta\in \Phi[\Gamma]$ denotes the action of $\W[\Gamma]$ on its root system by the canonical linear representation. Since by \cite{BellPar20} we know that these two actions coincide, we denote them here with the same notation.
\end{itemize}

\noindent In particular, \cite[Theorems 2.3 and 2.6]{BellParThiel} provide explicit group presentations for $\KVA[\Gamma]$ and $\PVA[\Gamma]$, respectively. Building on the Coxeter graph $\Gamma$, Bellingeri, Paris and Thiel  introduced another Coxeter graph $\hGamma$, defined as follows:
\begin{itemize}
    \item The vertex set of $\hGamma$ is $V(\hGamma)=\Phi[\Gamma]$.
    \item The Coxeter matrix $\widehat{\mathrm{M}}=(\hm_{\beta,\gamma})_{\beta,\gamma\in \Phi[\Gamma]}$ has entries $\hm_{\beta,\beta}=1$ for all $\beta\in \Phi[\Gamma]$, and $\hm_{\beta,\gamma}=m_{s,t}$ if there exists $w\in \W[\Gamma]$ and $s,t\in S$ such that $\beta=w(\alpha_s)$ and $\gamma=w(\alpha_t)$; and $\hm_{\beta,\gamma}=\infty$ otherwise. 
\end{itemize}
\noindent Sometimes we will call $\hGamma$ the \textit{root graph} associated with $\Gamma$. Theorem 2.3 in \cite{BellParThiel} demonstrates that the kernel $\KVA[
\Gamma]$ is isomorphic to the Artin group $\A[\hGamma]$. Notably, since the vertices of $\hGamma$ are in bijection with the roots in $\Phi[\Gamma]$, the Artin group $\A[\hGamma]=\KVA[\Gamma]$ is a finitely generated Artin group if and only if $\Gamma$ is of spherical type.\\

\subsection{Parabolic subgroups and colored Artin groups}\label{parabolic}

A fundamental concept in the theory of Coxeter and Artin groups theory is that of parabolic subgroups.
Let $\Gamma$ be a Coxeter graph with vertex set $S$. For any subset $X$ of $S$, let $\Gamma_X$ denote the full subgraph of $\Gamma$ spanned by $X$. A classical result shows that $\W[\Gamma_X]$ coincides with the subgroup $\langle X\rangle$ of $\W[\Gamma]$ generated by $X$. Similarly, Van der Lek \cite{van1983homotopy} showed that (for finitely generated Artin groups), $\A[\Gamma_X]$ coincides with $\langle X\rangle$ in $\A[\Gamma]$. These subgroups, called \textit{standard parabolic subgroups}, are denoted by $\W_X[\Gamma]$ and $\A_X[\Gamma]$, respectively. A \textit{parabolic subgroup} of $\W[\Gamma]$ or of $\A[\Gamma]$ is a conjugate of a standard parabolic subgroup. \medskip
\\
\noindent For $w \in \W[\Gamma]$, the word length of $w$ with respect
to the generating set $S$ is denoted by $l_S(w)$. If an expression of $w$ has length $l_S(w)$, it is called a \textit{reduced expression}. Reduced expressions are generally not unique. The set of generators $\{s_1,\ldots,s_n\}\subseteq S$ appearing in a reduced expression of an element $w=s_1\cdots s_n$ is called the \textit{support of $w$} and it is denoted by $supp(w)$. It can be shown that the support of an element $w$ does not depend on the choice of the reduced expression for $w$. \medskip \\
\noindent The following result characterizes a very special element that arises in spherical type Coxeter groups.
\begin{prop}\cite[Exercices - Chapitre 4]{Bourbaki}\label{proplongestelement} Let $\Gamma$ be a finite connected Coxeter graph with set of vertices $S$. Then $\W[\Gamma]$ is of spherical type if and only if there exists an element $w_0\in \W[\Gamma]$ such that for all $s\in S$, $l_S(w_0s)<l_S(w_0)$.
\end{prop}
\noindent Such an element $w_0$, when it exists, is called the \textit{longest element} of the Coxeter group $\W[\Gamma]$. It can also be shown that all the generators in $S$ must appear at least once in any expression of the longest element.\\
 Now, let us consider the cosets and double cosets with respect to the standard parabolic subgroups of $\W[\Gamma]$.
\begin{defn}
    Let $\Gamma$ be a Coxeter graph with set of vertices $S$, and let $X,Y\subset S$. We say that an element $w\in \W[\Gamma]$ is \textit{$(Y,X)$-minimal} if it is of minimal length in the double coset $\W[\Gamma_Y]\; w \;\W[\Gamma_{X}]=\{uwv\;|\; u\in \W[\Gamma_Y],v\in \W[\Gamma_X]\}$.
\end{defn}
\begin{lem}\cite[Exercices - Chapitre 4]{Bourbaki}
Let $\Gamma$ be a Coxeter graph with set of vertices $S$, and let $\W[\Gamma]$ be the associated Coxeter group. Take $w\in \W[\Gamma]$ and $X,Y\subset S$. Then there exists a unique $m_0\in \W[\Gamma]$ that is $(Y,X)-$minimal in the double coset $\W[\Gamma_Y]\; w \;\W[\Gamma_{X}]$. Moreover, there exist $u\in \W[\Gamma_Y]$ and $v\in \W[\Gamma_X]$ such that $w=um_0v$ and $l_S(w)=l_S(u)+l_S(m_0)+l_S(v)$.
\end{lem}

\noindent These notions will be useful in Sections \ref{inftyconnindec} and \ref{sectioncentralizerofKVA}. We now introduce another class of subgroups of Artin groups that will be essential for obtaining the indecomposability results.
\begin{defn}
    Let $\Gamma$ be a Coxeter graph on a set of vertices $S$, and let $\omega$ be the group homomorphism $\omega:\A[\Gamma]\longrightarrow \W[\Gamma]$ that sends the generator $\sigma_s\longmapsto s$ for each $s\in S$. The kernel of this homomorphism, denoted by $\CA[\Gamma]$, is called the \textit{colored Artin group} of $\Gamma$.
\end{defn}

\noindent Colored Artin groups are sometimes referred to as \textit{pure Artin groups}. Indeed, for the braid case $\B_n$, the homomorphism $\omega$ becomes the map that sends each braid to its associated permutation.
\begin{align*}
    \omega_{\B}\;:\;\B_n &\xrightarrow[\quad]{}\;\Sn\\
    \sigma_i &\longmapsto\; s_i.
\end{align*}
\noindent This map sends each classical generator $\sigma_i$ of the braid group $\B_n$ to the simple transposition $s_i=(i,\; i+1)$ of $\Sn$, for all $i\in \{1,\ldots,n-2\}$. The kernel $\ker(\omega_{\B})=\mathcal{PB}r
_n$ is called the \textit{pure} or \textit{colored braid group on $n$ strands}.\medskip\\
\noindent The surjection $\omega$
has a natural set-section $\varsigma : \W[\Gamma] \longrightarrow \A[\Gamma
]$. Take an element $w\in \W[\Gamma]$, and let $w =s_{i_1}\cdots s_{i_p}$ be an expression of $w$ such that $p=l_S(w)$ and $s_{i_j}\in S$ for all $j\in \{1,\ldots,p\}$. Namely, this expression is reduced. Define $\varsigma$ as $\varsigma(w):=\sigma_{s_{i_1}}\cdots \;\sigma_{s_{i_p}}$. It is clear that $\omega \circ \varsigma=\id_{\W[\Gamma]}$. We immediately observe that $\varsigma$ is not a group homomorphism because, for all $s\in S$, we have $\varsigma(s^2)=\varsigma(\id)=\id$ and $\varsigma(s)^2=\sigma_s^2\neq \id$.\medskip \\
\noindent
Indeed, the generators are involutions in the Coxeter group, but not in the Artin group. Nevertheless, this map is a well-defined set section of $\omega$, since by \cite{Tits1969LePD} its definition does not depend on the choice of the reduced expression for $w$. Moreover, even though $\varsigma$ is not a homomorphism, if $u , v \in \W[\Gamma]$ are such that $l_S(uv) = l_S(u) + l_S(v)$, then $\varsigma(uv)=\varsigma(u)\varsigma(v)$. \medskip \\
\noindent In this work, we will use a fundamental theorem by Godelle and Paris. Let $\Gamma$ be a finite Coxeter graph on a set of vertices $S$. Salvetti (see \cite{Sal87},\cite{Sal94}) constructed a simplicial complex $Sal(\Gamma)$ which has the same homotopy type as that of the complexified complement of the reflection arrangement of $\W[\Gamma]$. Specifically, $\pi_1(Sal(\Gamma))=\CA[\Gamma]$. This complexified complement of the reflection arrangement of $\W[\Gamma]$, is conjectured to be aspherical, giving rise to the famous \textbf{K$(\pi,1)$-conjecture} for (finitely generated) Artin groups. For a subset $X\subset S$, recall that we denote by $\Gamma_X$ the full subgraph of $\Gamma$ spanned by the vertices in $X$. The complex $Sal(\Gamma_X)$ naturally embeds in $Sal(\Gamma)$ with an embedding that is equivariant under the action of $\W_{X}=\W[\Gamma_X]$.

\begin{thm}[Theorem 2.2 in \cite{GodPar12}]\label{retraction}
    Let $\Gamma$ be a finite Coxeter graph on $S$, and let $X\subset S$. Then the natural embedding $i:Sal(\Gamma_X)\longrightarrow Sal(\Gamma)$ admits a retraction $p_X:Sal(\Gamma)\longrightarrow Sal(\Gamma_X)$ that is equivariant under the action of $\W_X$.
\end{thm}

\noindent By definition of a retraction, the restriction of such a $p_X$ to $Sal(\Gamma_X)$ is the identity. This result has several important consequences. In particular, it implies that if $\Gamma$ satisfies the $K(\pi,1)$ conjecture, then $\Gamma_X$ satisfies the $K(\pi,1)$-conjecture. In general, the conjecture is stated only for finite Coxeter graphs, because in the infinite case the Tits cone is not well defined. However, the simplicial complex $Sal(\Gamma)$, as well as its quotient $\salbar$ under the action of $\W[\Gamma]$, are defined also in the infinite case (see Section 3 in \cite{Gav24} for further information).
\begin{rmk}\label{godelleparisforinfinity} The proof of Theorem \ref{retraction} does not require the finiteness of the set of generators $S$, even if it is not explicitly mentioned. Indeed, the proof only uses standard parabolic subgroups of $\W[\Gamma]$ of spherical type. Therefore, we can extend the result also to the case in which the set of vertices $S$ is infinite. Namely, for all $X\subset S$, we can state that there exists a retraction
\[p_X: Sal(\Gamma)\longrightarrow Sal(\Gamma_X)\]
to the natural inclusion $\iota: Sal(\Gamma_X)\longrightarrow Sal(\Gamma)$. By considering the induced homomorphism on the fundamental groups, we obtain that there exists a group homomorphism, that we denote again by $p_X$
\[ p_X:\CA[\Gamma]\longrightarrow \CA[\Gamma_X]\]
such that its restriction to $\CA[\Gamma_X]<\CA[\Gamma]$ is the identity.
\end{rmk}

\noindent The group homomorphism $p_X$ is actually the restriction to $\CA[\Gamma]$ of a set retraction $\A[\Gamma]\longrightarrow\A[\Gamma_X]$ for the inclusion $i: \A[\Gamma_X]\longrightarrow \A[\Gamma]$, which was introduced implicitly in \cite{CharPar} and explicitly in \cite{BluPar23}.\medskip \\
\noindent This set retraction $p_X:\A[\Gamma]\longrightarrow \A[\Gamma_X]$ is described as follows. Let $g$ be in $\A[\Gamma]$, and write it as $g=\sigma_{s_{1}}^{\varepsilon_{1}}\,\cdots \;\sigma_{s_{{p}}}^{\varepsilon_{p}}$ with $s_1,\ldots, s_p\in S$ and $\varepsilon_j\in \{\pm 1\}$ for all $j\in \{1,\ldots,p\}$. We set $u_0 = \id \in \W[\Gamma]$ 
and, for $j \in \{1,\ldots, p\}$, we set $u_j = s_{{1}}^{\varepsilon_{1}}\,\cdots \,s_{{j}}^{\varepsilon_{j}}\in \W[\Gamma]$. We write each $u_j$ in the form $u_j = v_j w_j$
where $v_j \in \W[\Gamma_X]$ and $w_j$ is $(X, \emptyset)$-minimal. Let $j \in \{1,\ldots, p\}$. We set $t_j = w_{j-1} s_{j} w_{j-1}^{-1}$ if $\varepsilon_j = 1$, and
$t_j = w_j s_{j} w_j^{-1}$ if $\varepsilon_j =-1$. If $t_j \notin X$, then we set $\gamma_j = \id$. If instead $t_j \in X$, then set $\gamma_j= \sigma_{t_j}^{\varepsilon_j}$. Finally, we set
\[ p_X(g)=p_X(\sigma_{s_{1}}^{\varepsilon_{1}}\,\cdots\; \sigma_{s_{p}}^{\varepsilon_{p}}) = \gamma_1\,\cdots\; \gamma_p\in \A[\Gamma_X]. \]
\noindent In \cite[Proposition 2.3]{BluPar23} the authors show that $p_X$ is well defined (i.e., it does not depend on the choice of a reduced expression of $g$) and that it is a group homomorphism when restricted to $\CA[\Gamma]\longrightarrow\CA[\Gamma_X]$.
\begin{rmk}
    The set retraction $p_X:\A[\Gamma]\longrightarrow \A[\Gamma_X]$ as defined above has been introduced for finite Coxeter graphs. Nevertheless, since it is defined on every $g\in \A[\Gamma]$ and each element has a finite support, $p_X$ is well defined also for an infinite countable $S$.
\end{rmk}

\noindent The construction of $p_X$ is used to establish the following result.

\begin{prop}\label{propcarina}
Let $\Gamma$ be a Coxeter graph with a countable set of vertices $S=\{s_n\,|\,n\in \mathbb{N}_{>0}\}$. For all $n\in \mathbb{N}_{>0}$, let $X_{n}=\{s_1,\ldots,s_{n}\}$, and let $a_n$ be an element of $\A[\Gamma_{X_n}]$ of the form $a_n=\sigma_{s_1}^k\,\cdots\; \sigma_{s_n}^k$ for some $k\in \mathbb{Z}$, $k\neq 0$. Then $a_n\notin\A[\Gamma_{X}]$ for any $X\subset X_n$, $X\neq X_n$.
    \end{prop}
\noindent An interesting application of this proposition for $k=2$ will be used to show Lemma \ref{centralizerInfEdge}.

\begin{proof}
    We analyze separately the cases where $k$ is even and odd.\\
    Suppose that $k=2l$ is even. Then $a_n=\prod_{i=1}^n\sigma_{s_i}^{2l}$ belongs to $\CA[\Gamma]$, because its image under $\omega$ satisfies $\omega(a_n)=s_1^{2l}\,\cdots\, s_n^{2l}=\id$ in the Coxeter group $\W[\Gamma]$. Now, suppose $X=X_n\backslash \{s_j\}$ for some $j\in \{1,\ldots,n\}$, and assume $a_n\in \A[\Gamma_X]$. This implies $a_n \in\CA[\Gamma]\cap \A[\Gamma_X]=\CA[\Gamma_{X}]$. Now, consider the retraction homomorphism  
    \[p_X: \CA[\Gamma]\longrightarrow\CA[\Gamma_X],\]
    which is the identity on $\CA[\Gamma_X]$. Thus, $p_X(\sigma_{s_i}^{2l})=\sigma_{s_i}^{2l}$ if $i\neq j$. Furthermore, since we supposed that $a_n\in \CA[\Gamma_X]$, we also have that $p_X(a_n)=a_n$. By the construction of $p_X$, we easily see that $p_X(\sigma_{s_j}^{2l})=\id$. We now compute $p_X(a_n)=a_n$. Since $p_X$ is a group homomorphism, the last equality implies $p_X(\sigma_{s_1}^{2l})\;\cdots\;p_X(\sigma_{s_n}^{2l})=\sigma_{s_1}^{2l}\,\cdots\; \sigma_{s_n}^{2l
}$.   
Using the fact that $p_X(\sigma_{s_j}^{2l})=\id$, we can write:
\[
\sigma_{s_1}^{2l}\,\cdots\;\sigma_{s_{j-1}}^{2l} \,\sigma_{s_{j+1}}^{2l}\,\cdots\;\sigma_{s_n}^{2l
}=\sigma_{s_1}^{2l}\,\cdots\;\sigma_{s_{j-1}}^{2l}\,\sigma_{s_j}^{2l} \;\sigma_{s_{j+1}}^{2l}\,\cdots\;\sigma_{s_n}^{2l
}.
\]
Canceling out the common terms, we obtain $\sigma_{s_j}^{2l}=\id$, which is absurd. Therefore, $a_n$ cannot belong to any parabolic subgroup $\A[\Gamma_X]$ on a proper subset $X\subset X_n$.\medskip\\ 
\noindent Now let $k$ be odd, and suppose that $a_n$ belongs to a standard parabolic subgroup $\A[\Gamma_{X}]$ with $X$ strictly included in $X_n$. Consider the image of $a_n$ under the homomorphism $\omega$:
\begin{align*}
\omega\; :\; \A[\Gamma]\; &\xrightarrow[]{\qquad}\, \W[\Gamma],\\
a_n=\prod_{i=1}^n\sigma_{s_i}^{k}\,& \;\longmapsto\; \;\prod_{i=1}^n{s_i}^{k}.
\end{align*}
\noindent If $a_n\in \A[\Gamma_X]$, then $\omega(a_n)$ belongs to $\W[\Gamma_X]$. Since $k=2l+1$ is odd and in $\W[\Gamma]$ all the generators are involutions, the image $\omega(a_n)$ is $s_1s_2\,\cdots \,s_n$, which is a \textit{Coxeter element} of the Coxeter group $\W[\Gamma_{X_n}]$. It is known that Coxeter elements do not belong to any proper parabolic subgroup (see \cite[Theorem 3.1]{Par07}), so in particular $ \omega(a_n)\notin \W[\Gamma_X]$, which is a contradiction.\\
Therefore, $a_n$ never belongs to a standard parabolic subgroup with set of generators strictly contained in the set of generators appearing in the expression of $a_n$.

\end{proof}

\section{Infinity-connection and indecomposability}\label{inftyconnindec} 

In this section, we introduce a property called $\infty$-connection for Coxeter graphs $\Gamma$. We analyze these graphs, distinguishing between the cases of a finite or infinite number of vertices. In Subsection \ref{fingencase} we establish in Theorem \ref{SfinAindec} that if $\Gamma$ is an $\infty$-connected Coxeter graph on a finite set of vertices $S$, then the Artin group $\A[\Gamma]$ is indecomposable. This result relies on techniques involving Bass-Serre theory for groups acting on trees and amalgamated products of groups. To provide context, we also recall the main tools used in our arguments within the same subsection.\\
\\
\noindent In Subsection  \ref{infgencase}, we generalize Theorem \ref{SfinAindec} to some cases with an infinite set of generators. In particular, we consider the case in which the set $S$ admits a filtration $X_1\subset X_2 \subset\cdots \subset X_n\subset \cdots$ such that $\bigcup_{n=1}^{\infty} X_n=S$ and each $\Gamma_{X_n}$ is finite and $\infty$-connected. To achieve this result, we study the colored Artin group $\CA[\Gamma]$ associated with such a Coxeter graph $\Gamma$, which we prove to be indecomposable. Building on this, and using Nuida's result on the indecomposability of infinitely generated Coxeter groups, we conclude that $\A[\Gamma]$ is also indecomposable (Theorem \ref{SinfAindec}) when these conditions on $\Gamma$ hold. \medskip
\\
\noindent Later in Section \ref{kvaindec}, we use the $\infty$-connection property to show that the kernel $\KVA[\Gamma]$ is indecomposable.

\begin{defn}
Let $\Gamma$ be a Coxeter graph with a set of vertices $V(\Gamma)=S$. We say that $\Gamma$ is $\infty$-\textit{connected} if the subgraph $\Gamma^{\infty}$ obtained from $\Gamma$ by deleting all the edges labeled by integers $m_{s,t}< \infty$, is connected.
\end{defn}
\noindent Observe that such a Coxeter graph is always connected (i.e., $\W[\Gamma]$ is \textit{irreducible}) and can never be of spherical type unless $|V(\Gamma)|=1$.

\subsection{Finitely generated case}\label{fingencase}

\noindent
The aim of this subsection is to prove the following result.
\begin{thm}\label{SfinAindec}
    Let $\Gamma$ be an $\infty$-connected Coxeter graph on a finite set of vertices $S=\{s_1,\ldots,s_n\}$. Then, the Artin group $\A[\Gamma]$ is indecomposable.
\end{thm}
\noindent
 We point out that, aside from $\Gamma$ of spherical type, little is known about the indecomposability of $\A[\Gamma]$. This last case was studied by Paris in \cite{Par04}, where it is proven that irreducible Artin groups of spherical type are indecomposable. In Section \ref{kvaindec}, we will show that for any Coxeter graph $\Gamma$, the graph $\widehat{\Gamma}$ (whose vertices correspond bijectively with the roots in $\Phi[\Gamma]$) satisfies the property of $\infty$-connection. Thus, when $\hGamma$ has a finite number of vertices, namely when $\Gamma$ is of spherical type, we can apply Theorem \ref{SfinAindec} to conclude that $\A[\hGamma]=\KVA[\Gamma]$ is indecomposable.\\
\\
\noindent The proof of this theorem requires some preliminaries concerning the centralizers of specific elements in the group. Recall that if $G$ is a group and $E$ is a subset of $G$, the \textit{centralizer} of $E$ in $G$ is the subgroup $Z_{G}(E)=\{g\in G\,|\, ge=eg\,\,\; \mbox{for all }e \in E\}< G$. If $E$ is a singleton $\{e\}$, we write $Z_{G}(e)$ to mean the centralizer of $e$ in $G$. The \textit{center} of a group $G$ is the subgroup $Z(G)=\{g\in G\,|\, gh=hg\,\mbox{for all }h\in G\}$. Clearly, the center of a group is contained in the centralizer of any of its subsets. \\
\\
\noindent 
Now, we will consider a specific element in a finitely generated Artin group and attempt to describe its centralizer. 

\begin{lem}\label{centralizerinA}
Let $\Gamma$ be a connected Coxeter graph on a finite set of vertices $S=\{s_1,\ldots,s_n\}$, with $\Gamma$ $\infty$-connected. Let $\theta$ be $\theta=\sigma_{s_1}^2\,\sigma_{s_2}^2\,\cdots\, \sigma_{s_n}^2\in \A[\Gamma]$. Then, the centralizer $Z_{\A[\Gamma]}(\theta)$ is the infinite cyclic subgroup $\langle \theta \rangle$ generated by $\theta$.
\end{lem}

\noindent The proof of this lemma uses techniques from Bass-Serre theory, which can be found in detail in \cite{Serre} or more concisely in \cite[Chapter VII]{Baums93}. We outline here the main tools and establish the notations that will be used throughout this work.\\\\
\noindent
Across the article, by a \textit{tree} $T$ we will mean a simplicial tree, whose edges and vertices will be denoted respectively by $E(T)$ and $V(T)$. Recall that an action $G\times T\longrightarrow T$ of a group $G$ on a tree $T$ is said to be
\textit{without inversion} if for every edge $e\in E(T)$ and every $g\in G$, the equality $ge=e$ implies that $g$ fixes the two extremities of $e$.\\
\\
\noindent
Let $G$ be group that acts without inversion
 on a tree $T$. An element $h\in G$ is said to be \textit{hyperbolic} if for all $x\in V(T)$, $hx\neq x$. Observe that there is a natural distance $d:V(T)\times V(T)\longrightarrow \mathbb{N}$ on $T$, where each edge has length 1, and the distance between two vertices $x$ and $y$ in $V(T)$ is the length of the unique geodesic in $T$ joining $x$ and $y$. If $h\in G$ is a hyperbolic element, we define the following objects: 
\begin{align*}
    &\lambda_h=min\{d(x,hx)\,|\, x\in V(T)\}; &\M_h=\{x\in V(T)\,|\, d(x,hx)=\lambda_h\}.
\end{align*}
\noindent 
Let $\Lm_h$ be the full subgraph of $T$ spanned by $\M_h$. The following result, originally due to Tits, can be found in \cite[Proposition 24]{Serre}. It provides a characterization of the set $\Lm_h$ for each $h$ hyperbolic element. \medskip
\\
\noindent When $L$ is a graph, we will say that $L$ is a \textit{line} if all of its vertices have valence 2, namely if its geometric realization is
homeomorphic to $\mathbb R$.
\begin{prop}[Proposition 3.2 in \cite{Tits1970}]\label{linehyper} Let $G$ be a group acting without inversion on a tree $T$, and let $h\in G$ be a hyperbolic element. Then $\Lm_h$ is a line, and $h$ acts on it as a translation of amplitude $\lambda_h$.
\end{prop}
\noindent Actually, as we can see in \cite[Proposition 25]{Serre}, the fact of acting as a translation on an invariant line characterizes the hyperbolic elements.
\begin{prop}\cite[Proposition 25]{Serre} \label{linehyperIFF}
Let $G$ be a group acting without inversion on a tree $T$, and let $h$ be an element of $G$. Then $h$ is hyperbolic if and only if there is a line $\Lm$ in $T$, stable under $h$, on which $h$ acts as a translation of non-zero amplitude.
\end{prop}
\noindent The implication that does not coincide with Proposition \ref{linehyper} is a consequence of the study of the fixed points of an automorphism of a tree (see \cite[Section 6.4]{Serre} for more details).
\begin{nt}
    Given $G$ a group acting without inversion on
a tree $T$ and $h \in G$ a hyperbolic element, the line $\mathrm{L}_h$ described above will be called \textit{axis of the element} $h$.
\end{nt}
\noindent This characterization of hyperbolic elements is particularly useful for computing centralizers in a group acting (without inversion) on a tree.
\begin{lem}\label{centralhype}
Let $G$ be a group acting without inversion on a tree $T$, and let $h\in G$ be a hyperbolic element. If $g\in Z_{G}(h)$, then $g$ leaves the axis $\Lm_h$ invariant.
\end{lem}
\begin{proof}
    Suppose that $g\in Z_{G}(h)$, meaning $ghg^{-1}=h$ and then that $\mathrm{L}_h=\mathrm{L}_{ghg^{-1}}$. Since $\mathrm{L}_{ghg^{-1}}=g(\mathrm{L}_{h})$, we can conclude that $\mathrm{L}_{h}=g(\mathrm{L}_{h})$.
\end{proof}
\noindent We will apply this result with $G$ an amalgamated product of groups and $T$ a specific tree. For an in-depth treatment of amalgams, see the references by Serre \cite{Serre} and Baumslag \cite{Baums93}. Here, we simply recall the fundamental structural results about amalgamated products.

\begin{thm}[\cite{Serre}]
    \label{normalformamalgam}
Let $A_1,\ldots, A_p,C$ be a collection of groups. We suppose
that $C$ is a subgroup of $A_j$ for all $j\in \{1, \ldots, p\}$, and consider the amalgamated
product \[G = A_1\;\frpp_C\; A_2\; \frpp_C\; \cdots\; \frpp_C\; A_p.\] For each $j \in \{1,\ldots,p\}$, we choose a transversal
$T_j$ of $C$ in $A_j$. Then each element $g\in G$ can be uniquely written in the form
$g =t_1\,t_2\, \cdots \,t_k\, c$ such that:
\begin{itemize}
    \item[1)] $c\in C$, and for each $i \in \{1,\ldots,k\}$, there exists $j = j(i) \in \{1, \ldots,p\}$ such that
$t_i \in T_j^* = T_j\, \backslash \,\{\id\}$,
\item[2)] $j(i)\neq j(i + 1)$ for all $i\in \{1, \ldots,k-1\}$.
\end{itemize}
In particular, we have $g \in C$ if and only if $k = 0$ and $g = c$.
\end{thm}
\noindent The expression of $g$ given above is called the \textit{normal form} of $g$.\medskip \\
 \noindent In the case of only two factors, there is a handy consequence of the previous theorem.
\begin{lem}[\cite{Serre}]\label{alternformamalg}
Let $G=A\,\frpp_C\, B$, where $A,B$ and $C$ are groups and $C$ is a subgroup of both $A$ and $B$. Then each $g\in G\,\backslash\, C$ has an expression $g=m_1\;\cdots\; m_n$, such that:
\begin{itemize}
    \item For each $i\in \{1,\ldots,n\}$, there exists $j=j(i)\in\{1,2\}$ such that $m_i\in D_{j(i)}\backslash C$, where $D_1=A$ and $D_2=B$. 
    \item For all $i\in \{1,\ldots,n-1\}$, $j(i)\neq j(i+1)$.
\end{itemize} 
Furthermore, the integer $n>0$ depends only on $g$, that is, any two such representations of $g$ have the same number of factors.
\end{lem}

\noindent An expression of $g$ as in the previous lemma is sometimes called a \textit{strictly alternating} $(A\backslash C)\sqcup(B\backslash C)$-product. \medskip\\
\noindent Let $G=A\,\frpp_C \, B$, where $A,B$ and $C$ are groups, and $C$ is a subgroup of both $A$ and $B$. Then $G$ acts naturally without inversion on a tree $T$, described as follows:
\begin{itemize}
    \item The set of vertices $V(T)$ of $T$ is the disjoint union of the sets of left cosets $\{gA\,|\, g\in G\}$ and $\{gB\,|\,g\in G\}$, corresponding to $A$ and $B$, respectively.
    \item The set of edges $E(T)$ of $T$ is the set of left cosets $\{gC\,|\, g\in G\}$, corresponding to $C$.
    \item For each $g\in G$, the extremities of the edge $gC$ are $gA$ and $gB$.
\end{itemize}
\begin{figure}[h!]
    \centering \includegraphics[width=13cm]{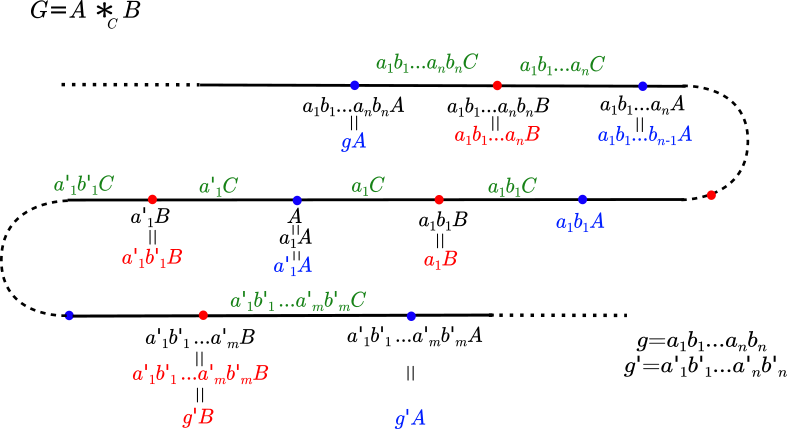}
    \caption{The tree $T$ associated with an amalgamated product of groups $G=A\,\frpp_{C}\,B$.}
    \label{generaltreefig}
\end{figure}
\noindent Further information can be found in \cite[Chapter VII]{Baums93}. In Figure \ref{generaltreefig} we illustrate as an example the unique geodesic path in $T$ joining the vertices $gA$ and $g'B$, with $g=a_1b_1\,\cdots \,a_n b_n$, $g'=a_1'b_1'\,\cdots \,a_m'b_m'\in G$ and $a_i,a_j'\in A\backslash C$, $b_i,b_j'\in B\backslash C$ for all $i\in \{1,\ldots,n\}$ and all $j\in \{1,\ldots,m\}$. Note that $G$ acts on $T$ by left multiplication. \\
\\
\noindent In the next lemma, we consider an amalgamated product of groups acting on its tree as we described above, and we show that an element of a certain form is always hyperbolic.

\begin{lem}\label{abhyperbolic}
Let $A,B$ be two groups, and let $C$ be a subgroup of both. Consider the amalgamated product $G=A\, \frpp_{C} \,B$, and let $h$ be an element of the form $h=ab$, with $a\in A\backslash C$ and $b\in B\backslash C$. Then, $h$ is a hyperbolic element acting on the tree $T$ associated with the amalgamated product $A\,\frpp_C \,B$. Moreover, $\mathrm{L}_{h}$ is the subgraph of $T$ spanned by the vertex set $\bigcup_{l\in \mathbb Z} \{h^lA, h^l B\}$,
and $h$ acts on $\mathrm{L}_h$ amplitude of 2.
\end{lem}
\begin{proof}
    Thanks to Proposition \ref{linehyperIFF}, to show that $h$ is hyperbolic, it suffices to show that there exists a line $\Lm$ in $T$ such that $\Lm$ is stable under $h=ab$, and $h$ acts on such a path as a translation of non-zero amplitude. \medskip \\
    \noindent Consider the subgraph $\Lm$ of $T$ such that the vertices $V(\Lm)$ are $\{h^lA\,|\,l\in \mathbb{Z}\}\sqcup \{h^lB\,|\, l\in \mathbb{Z}\}$. In Figure \ref{pathhypELM} we illustrate as an example the path in $\Lm$ joining the vertices $h^{-2}A$ and $h^2B$. It is clear that $h$ leaves $\Lm$ invariant.
    \begin{figure}[h!]
        \centering
        \includegraphics[width=16cm]{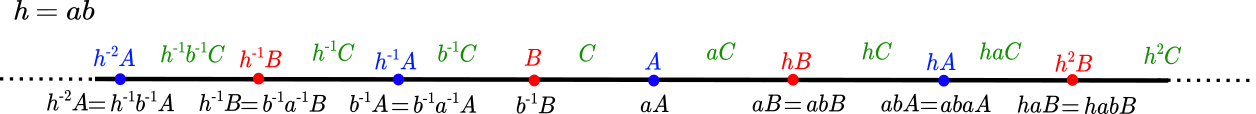}
        \caption{The path joining $h^{-2}A$ and $h^2B$.}
        \label{pathhypELM}
    \end{figure}
    \noindent In a graph $X$, the shortest path joining any two vertices is called a \textit{geodesic}. \medskip\\
    \noindent By \cite[Proposition I.8]{Serre}, we know that in a tree, a geodesic is unique and forms a line. In this context, a path is a geodesic (i.e. the shortest path between two vertices) if it does not contain back-trackings. We recall that a path whose vertices are $V_0, V_1,\ldots ,V_i,\ldots,V_n$ has a back-tracking if there exists an $i\in \{1,\ldots,n-1\}$ such that $V_{i-1}=V_{i+1}$. \medskip\\ 
\noindent We aim to show that the path whose vertices are $h^{-1}A$, $B$, $A$, $hB$ is the geodesic joining $h^{-1}A$ to $hB$, meaning it contains no back-tracking.  \medskip \\
\noindent First, we verify that these four vertices are distinct. In fact, two vertices labeled by two cosets corresponding to different subgroups of $G$, can never coincide. Therefore, $B$ and $A$ are necessarily distinct, $B$ does not coincide with $h^{-1}A$, and $A$ does not coincide with $hB=abB=aB$.\medskip\\ \noindent Now, suppose for contradiction that $B=hB=abB=aB$. This would imply $a\in B$, contradicting our assumption that $a\in A\backslash C=A\backslash (A\cap B)$. In the same way we show that the vertices $h^{-1}A$ and $A$ are distinct, hence, the path from  $h^{-1}A$ to $hB$ is a geodesic. \medskip \\
    \noindent Next, to show that $\Lm$ is a line, we verify that it does not contain back-trackings. Observe that, if $\Lm$ has a back-tracking in $h^lA$ for some $l\in \mathbb{Z}$, then $(h^{-1}A\,,B,\,A\,,hB)$ has a back-tracking in $A$. Similarly, if $\Lm$ has a back-tracking in $h^lB$ for some $l\in\mathbb{Z}$, then $(h^{-1}A\,,B,\,A\,,hB)$ has a back-tracking in $B$. Since $(h^{-1}A\,,B,\,A\,,hB)$ has no back-tracking, it follows that $\Lm$ has no back-tracking, that is, $\Lm$ is a line.
   \medskip \\
    \noindent The fact that $h$ acts on $\Lm$ as a  translation of non-zero amplitude, is now clear. We can then conclude, thanks to Proposition \ref{linehyperIFF}, that $h=ab$ is hyperbolic.\medskip\\
    \noindent Now remark that if $h$, $A,B,C$ and $T$ are as above, then  \[\lambda_{h}=min\{d(x, h x)\,|\, x\in V(T) \}=2.\] Indeed, by the construction of $T$, if $x$ is a vertex of the form $gA$, its adjacent vertices are of the form $g'B$, for some $g'\in G$. Since the action of $h$ cannot send a coset with respect to $A$ to one with respect to $B$, the distance $d(x, h x)$ must be at least two. In fact, this bound is realized: for instance, the vertex $A=\id A$ is mapped by $h=ab$ to $abA=hA$, which is at distance 2 from $A$, as shown in Figure \ref{pathhypELM}. Thus, $\lambda_{h}=2$, and $h$ acts on the axis $\Lm_{h}$, that is the full subgraph spanned by $\M_h=\{x\in V(T)\,|\, d(x,h x)=\lambda_{h}\}$, as a shift by two places. Additionally, the vertices of $\Lm_h$ are \[V(\Lm_h)=\M_h=\{h^lA\,|\, l\in \mathbb{Z}\}\sqcup \{h^lB\,|\, l\in \mathbb{Z}\},\] and the edges are $E(\Lm_h)=\{h^lC\,|\, l\in \mathbb{Z}\}\sqcup \{h^laC\,|\, l\in \mathbb{N}\}\sqcup \{h^{-l}b^{-1}C\,|\, l\in \mathbb{N}\}$.
\end{proof}

\begin{rmk}\label{lambdaofhyperAB}
Observe that since $h^{-l}b^{-1}=h^{-l}b^{-1}a^{-1}a=h^{-l-1}a$, the edges $E(\Lm_h)$ can be expressed more concisely as \[E(\Lm_h)=\{h^lC\,|\, l\in \mathbb{Z}\}\sqcup \{h^laC\,|\, l\in \mathbb{Z}\}.\]
    
\end{rmk}
\begin{rmk}\label{powerofahyp}
Let $h$ be a hyperbolic element acting on a tree.
If $h$ is hyperbolic, then so is $h^l$ for any $l\neq 0$. Furthermore, $\Lm_h = \Lm_{h^l}$ and $\lambda_{h^l} = |l|\,\lambda_h $.
\end{rmk}

\noindent This construction will be used to describe the centralizer of certain special elements in a finitely generated Artin group $\A[\Gamma]$ which has at least one edge of $\Gamma$ labeled by $\infty$. Indeed, as shown in the next result, Artin groups with an $\infty$-labeled edge can be expressed as amalgamated products.

\begin{lem}\cite[Lemma 3.3]{BellPar20} \label{Artingroupsareamalgamated}
Let $\Gamma$ be a Coxeter graph with a finite set of vertices $S$. Let $X$ and $Y$ be two subsets of $S$ such that $X\cup Y=S$ and $m_{s,t}=\infty$ for all $s\in X\backslash (X\cap Y)$ and $y\in Y\backslash (X\cap Y)$. Then 
\[ \A[\Gamma]=\, \A[\Gamma_{X}]\,\frpp_{\A[\Gamma_{X\cap Y}]}\,\A[\Gamma_Y].\]
\end{lem}

\noindent We are now ready to study the centralizers of certain hyperbolic elements in an Artin group with at least one $\infty$-labeled edge.

\begin{lem}\label{centralizerInfEdge}
 Let $\Gamma$ be a connected Coxeter graph on a finite set of vertices $S=\{s_1,\ldots,s_n\}$. Let $i,j\in \{1,\ldots,n\}$ be such that $i\neq j$ and $m_{s_i,s_j}=\infty$. Let $\theta$ be $\theta=\sigma_{s_1}^2\,\sigma_{s_2}^2\,\cdots\, \sigma_{s_n}^2\in \A[\Gamma]$. Let $h\in \A[\Gamma]$ belong to the centralizer $Z_{\A[\Gamma]}(\theta^k)$ for some $k\in \mathbb Z$. Then, $h$ has the form $h=\theta^{kl}\,c$, where $l\in \mathbb{Z}$ and $c\in \A[\Gamma_{S\backslash\{s_i,s_j\}}]$.   
\end{lem}
\begin{proof}
    Define the sets $X:=S\backslash \{s_j\}$, $Y:=S\backslash \{s_i\}$ and $Z:=X\cap Y=S\backslash\{s_i,s_j\}$. Set now $A=\A[\Gamma_X]$, $B=\A[\Gamma_{Y}]$ and $C=\A[\Gamma_Z]$. Thanks to Lemma \ref{Artingroupsareamalgamated}, we know that $G:=\A[\Gamma]=A\,\frpp_{C}\,B$. Modulo swapping the roles of $s_i$ and $s_j$, we can suppose that $i<j$.\medskip \\
\noindent As noted earlier, $G=\A[\Gamma]$ acts naturally by left multiplication on a tree $T$, where the vertices are $V(T)=\{gA\,|\, g\in G\}\sqcup \{gB\,|\, g\in G\}$, and the edges are $E(T)=\{gC\,|\, g\in G\}$. The two extremities of each edge $gC$ are $gA$ and $gB$.\medskip \\
\noindent Now, consider the element $\theta=\sigma_{s_1}^2\cdots \,\sigma_{s_i}^2\, \sigma_{s_{i+1}}^2 \cdots \, \sigma_{s_j}^2 \cdots\, \sigma_{s_n}^2$. We first show that $\theta$ is hyperbolic (acting) on $T$ and then use Lemma \ref{centralhype} to describe explicitly its centralizer.\medskip \\
\noindent Define $\sigma_{s_1}^2\cdots\, \sigma_{s_i}^2=a\in A$ and $\sigma_{s_{i+1}}^2\cdots\; \sigma_{s_{n}}^2=b\in B$. We can then rewrite $\theta$ as $\theta=a b$. By Lemma \ref{abhyperbolic}, to conclude that $\theta=ab$ is hyperbolic, it suffices to show that $a\in A\backslash C$ and $b\in B\backslash C$. Consider the element $a$, that is the product of all the squared generators of the standard parabolic subgroup $\A[\Gamma_{\{s_1,\ldots,s_i\}}]<\A[\Gamma_X]=A$. Proposition \ref{propcarina} ensures that $a$ does not belong to any standard parabolic subgroup of $\A[\Gamma_{\{s_1,\ldots,s_i\}}]$ whose set of generators is strictly contained in $\{\sigma_{s_1},\ldots,\sigma_{s_i}\}$. In particular, $a\notin \A[\Gamma_{\{s_1,\ldots,s_{i-1}\}}]$. Suppose now that $a\in C=\A[\Gamma_{S\backslash\{s_i,s_j\}}]$. By \cite{van1983homotopy}, this would imply that $a\in C\cap \A[\Gamma_{\{s_1,\ldots,s_i\}}]=\A[\Gamma_{\{s_1,\ldots, s_{i-1}\}}]$, which we have just shown to be false. Therefore, $a\notin C$. A similar argument shows that  $b\notin C=\A[\Gamma_{S\backslash \{s_i,s_j\}}]$. Thus, $\theta=ab\in A\,\frpp_{C}\, B$ is hyperbolic. By Remark \ref{powerofahyp}, we also have that $\theta^k$ is hyperbolic for any $k\in \mathbb Z$.\medskip
\\
\noindent As in Lemma \ref{abhyperbolic}, the integer $\lambda_{\theta}$ equals two, so that $\lambda_{\theta^k}=|k|\lambda_{\theta}=2|k|$, and $\theta^k$ acts on the axis $\Lm_{\theta^k}$, which is the full sub-tree spanned by $\M_{\theta^k}=\{x\in V(T)\,|\, d(x,\theta^k x)=\lambda_{\theta^k}\}$, as a shift by $2|k|$ places. In the same lemma, we also pointed out that the vertices of $\Lm_{\theta^k}$ are $V(\Lm_{\theta^k})=\M_{\theta^k}=\{\theta^{kl} A\,|\,l\in \mathbb{Z} \}\cup \{\theta^{kl} \,B\,|\, l\in \mathbb{Z}\}$, while the edges are $E(\Lm_{\theta^k})=\{\theta^{kl} C\,|\, l\in \mathbb{Z}\}\sqcup \{\theta^{kl}aC\,|\, l\in \mathbb{Z}\}$.\medskip
\\
\noindent As we saw in Lemma \ref{centralhype}, if $h\in G$ belongs to the centralizer of $\theta^k$, then it must preserve $\Lm_{\theta^k}$. This means that for each $x\in V(\Lm_{\theta^k})$ and each $e\in E(\Lm_{\theta^k})$, we must have $hx\in V(\Lm_{\theta^k})$ and $he\in E(\Lm_{\theta^k})$. The only possible automorphisms that preserve the axis $\Lm_{\theta^k}$ (which is a bi-infinite line in the tree) are either translations (of vector parallel to the axis) or reflections (with axis orthogonal to the axis).\medskip\\
\noindent Specifically, the edge labeled by the coset $C$ is mapped by $h$ to another edge of $\Lm_{\theta^k}$, which can be either of the form $\theta^{kl}C$ with $l\in \mathbb{Z}$, or $\theta^{kl}aC$ with $l\in \mathbb{Z}$. Consequently, $h$ must be of the form $h=\theta^{kl}c$ with $c\in C$ and $l\in \mathbb{Z}$, or $h=\theta^{kl}ac$ with $c\in C$ and $l\in \mathbb{Z}$. If $h=\theta^{kl}c$, then it acts on $\Lm_{\theta^k}$ as a translation of amplitude $2kl$. On the other hand, if $h=\theta^{kl}ac$, it acts on $\Lm_{\theta^k}$ as a reflection. \medskip \\
\noindent Now, suppose $h=\theta^{kl}ac$ with $l\in \mathbb{Z}$ and $c\in C$. Then the edge represented by the coset $C$ is sent by $h$ in the edge represented by $\theta^{kl}aC$. Moreover, observe that 
\begin{equation}\label{imagesofh}
    \begin{cases}
        
    &hC=\theta^{kl}aC,\\
    &hB=\theta^{kl}acB=\theta^{kl+1}B,\\
&hA=\theta^{kl}acA=\theta^{kl}A.
    \end{cases}
\end{equation}
\begin{figure}[h!]\centering\includegraphics[width=15cm]{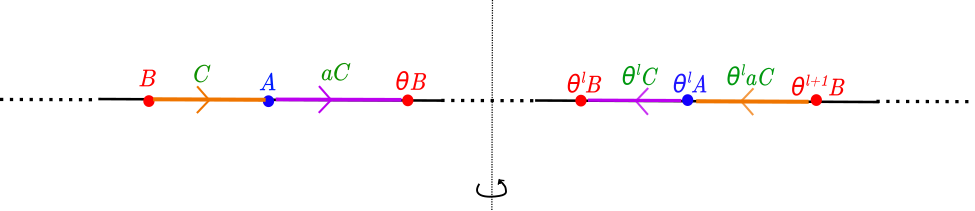}
    \caption{The action of $h=\theta^lac$ on $L_{\theta}$}
    \label{linereflection}
\end{figure}
\noindent The action of such an $h=\theta^{kl}ac$ is illustrated in Figure \ref{linereflection}. The edge $C$ and its image $hC$ are colored in orange.
\noindent
We want now to show that if $h$ is a reflection, we find an absurd. \medskip \\
\noindent Since $\theta^k|_{\Lm_{\theta^k}}$ is a translation, if we assume $h|_{\Lm_{\theta^k}}$ to be a reflection, we have:
\[
\theta|_{\Lm_{\theta^k}}\,=h\theta^k h^{-1}|_{\Lm_{\theta^k}}=h|_{\Lm_{\theta^k}} \;\cdot\theta^k|_{\Lm_{\theta^k}}\;\cdot h^{-1}|_{\Lm_{\theta^k}}=(\theta^k|_{\Lm_{\theta^k}})^{-1}=\theta^{-k}|_{\Lm_{\theta^k}}. 
\]
However, the previous equation cannot hold, because $\theta^k|_{\Lm_{\theta^k}}$ has infinite order.

\noindent
Therefore, $h$ cannot be a reflection, and the only possibility for $h$ to be in $Z_{\A[\Gamma]}(\theta^k)$ is  \[h=\theta^{kl}c,\]
for some $l\in \mathbb{Z}$ and some $c\in C=\A[\Gamma_{S\backslash\{s_i,s_j\}}]$.
\end{proof}
\noindent
This result is in particular true for $k=1$, and it will be crucial in analyzing the case where the Coxeter graph is $\infty$-connected. In particular, this structure implies that the associated Artin group can be written as an amalgamated product in multiple ways. As we will see in the proof of Lemma \ref{centralizerinA}, the expressions of $h$ in the centralizer of $\theta$ must be compatible with all such decompositions, forcing $h$ to belong to the infinite cyclic subgroup generated by $\theta$.

\begin{proof}[Proof of Lemma \ref{centralizerinA}]
Consider the Artin group $\A[\Gamma]$ and its associated Coxeter matrix $\M[\Gamma]$. For each $i,j\in \{1,\ldots,n\}$ where $i\neq j$, let the entry $m_{s_i,s_j}$ of $\M[\Gamma]$ be denoted simply by $m_{i,j}$. Since $\Gamma$ is $\infty$-connected, for any $s_i$ in $\{s_1,\ldots,s_n\}$, there exists a $j=j(i)\in \{1,\ldots, n\}$ such that $s_i\neq s_{j}$ and $m_{i,j}=\infty$. Without loss of generality, suppose that $j>i$.\medskip
\\
\noindent Now, define the sets $\mathcal{B}_j=S\backslash \{s_i\}$, $\mathcal{A}_i=S\backslash \{s_j\}$, and $\mathcal{C}_{ij}=\mathcal{A}_i\cap \mathcal{B}_j=S\backslash \{s_i,s_j\}$. We denote $\A[\Gamma_{\mathcal{A}_i}]$ by $A_i$, $\A[\Gamma_{\mathcal{B}_j}]$ by $B_j$, and $\A[\Gamma_{\mathcal{C}_{ij}}]$ by $C_{ij}$.
According to Lemma \ref{Artingroupsareamalgamated}, we know that we can write $G=\A[\Gamma]=A_i\,\frpp _{C_{ij}}\, B_j$.  \medskip \\
\noindent
We can now apply Lemma \ref{centralizerInfEdge} to deduce that, if $h\in Z_{\A[\Gamma]}(\theta)$, then $h=\theta^{l_{ij}}\,c_{ij}$ for some $l_{i,j}\in \mathbb{Z}$ and $c_{ij}\in C_{ij}$. We recall that $C_{ij}=\A[\Gamma_{\mathcal{C}_{ij}}]=\A[\Gamma_{S\backslash\{s_i,s_j\}}]$. The integer $l_{ij}$ and the element $c_{ij}$ depend on $i$ and $j=j(i)$. \medskip\\
\noindent Now, suppose there is another index $k\in \{1,\ldots,n\}$ such that $i<k$ and $m_{ik}=\infty$. Therefore, we have  $h=\theta^{l_{ij}}c_{ij}$ for $c_{ij}\in \A[\Gamma_{S\backslash\{s_i,s_j\}}]$, but also $h=\theta^{l_{ik}}c_{ik}$ for $c_{ik}\in \A[\Gamma_{S\backslash\{s_i,s_k\}}]$. So we get
\[
\theta^{l_{ij}}c_{ij}=\theta^{l_{ik}}c_{ik}\qquad \Longrightarrow \qquad \theta^{l_{ij}-l_{ik}}=c_{ik}c_{ij}^{-1}\in \A[\Gamma_{S\backslash \{s_i\}}]=B_j=B_k.
\]
As we observed in Remark \ref{powerofahyp}, if $l_{ij}-l_{ik}\neq 0$ then $\theta^{l_{ij}-l_{ik}}=(a_i b_j)^{l_{ij}-l_{ik}}$ acts on $\Lm_{\theta}$ as a translation of amplitude $l_{ij}-l_{ik}$ and it has no fixed points in $\Lm_{\theta}$. The element $c_{ik}c_{ij}^{-1}$ instead, belonging to $B_j$ (which coincides with $B_k$) fixes the vertex labeled by the coset $B_j$.
Therefore, the only possibility is 
$l:=l_{ij}=l_{ik}$.\medskip\\
\noindent Now, since $\Gamma$ is $\infty$-connected, we can conclude that for all $i\in \{1,\ldots,n\}$, all the integers $l_{ij(i)}$ must be equal. We set then $l=l_{ij(i)}$ for all $i$.\medskip
\\
\noindent Thus, we have shown that for any $1\leq i \leq n$, we can write $h\in Z_{\A[\Gamma]}(\theta)$ as $h=\theta^l c_{ij(i)}$, where $l\in \mathbb{Z}$ and $c_{ij(i)}\in C_{ij(i)}=\A[\Gamma_{\mathcal{C}_{ij(i)}}]$. In particular, we have  $c_{ij(i)}=\theta^{-l}h$ for all $i$, which implies that $c_{ij(i)}=c_{i'j(i')}$ for all $i,i'\in \{1,\ldots,n\}$. We set then $c_{ij}=c\in C_{ij(i)}$. But this must be true for all $i$, thus $c\in \bigcap_{i=1}^{n} C_{ij(i)}=\bigcap_{i=1}^{n} \A[\Gamma_{\mathcal{C}_{ij(i)}}]$, that by \cite{van1983homotopy} is $ \A[\Gamma_{\bigcap_{i=1}^{n}\mathcal{C}_{ij(i)}}]$. Now, since by construction $\bigcap_{i=1}^{n}\mathcal{C}_{ij(i)}=\emptyset$, we find $c=\id$. Finally, we showed that if $h\in Z_{\A[\Gamma]}(\theta)$, then $h=\theta^l$ for some $l\in \mathbb{Z}$. Hence, 
$Z_{\A[\Gamma]}(\theta)=\langle \theta \rangle$.
\end{proof}

\noindent As announced before, the previous result will be the key to prove the indecomposability of the Artin group $\A[\Gamma]$. \medskip \\
\noindent The following lemma is very easy to prove, but it leads to very useful results, and it will be used during the whole work, so we include here its proof.

\begin{lem}\label{lemmadecompcentralizer}
    Let $G,A,B$ be three groups such that $G=A\times B$, and let $E$ be a subset of $G$. Then, the centralizer of $E$ decomposes as $Z_G(E)=(Z_G(E)\cap A)\times( Z_G(E)\cap B)$.
\end{lem}
\begin{proof}
    The inclusion $(Z_G(E)\cap A)\times (Z_G(E)\cap B )\subseteq Z_G(E)$ is obvious, so we move to analyze the other direction. Take $g\in Z_G(E)$. By the decomposition of $G$, we can write $g=g_Ag_B$, with $g_A\in A$, $g_B\in B$ and $[g_A,g_B]=1$. Similarly, for each $e\in E\subset G$, we can write $e=e_Ae_B=e_Be_A$ with $e_A\in A$ and $e_B\in B$. Now we write
    \[(g_Ae_A\,,\,g_Be_B)=ge=eg=(e_Ag_A\,,\,e_Bg_B).\]
    Thus, $g_Ae_A=e_Ag_A$. Moreover, since $g_A\in A$ and $e_B\in B$, $g_Ae_B=e_Bg_A$. Therefore, $g_A\in Z_G(E)\cap A$. In the same way, we obtain $g_B\in Z_G(E)\cap B$, which shows that $Z_G(E)\subseteq (Z_G(E)\cap A)\times( Z_G(E)\cap B)$, and the requested equality.
\end{proof}
\noindent We are now ready to show the main result of this section. The proof is based on the fact that if we can find in $\A[\Gamma]$ an element with infinite cyclic centralizer, then the Artin group is indecomposable.

\begin{proof}[Proof of Theorem \ref{SfinAindec}]
The proof of this theorem relies on Lemma \ref{centralizerinA}. Take $\theta$ to be the product of all the squared generators $\theta=\sigma_{s_1}^2\cdots \,\sigma_{s_n}^2$, and we know that $Z_{\A[\Gamma]}(\theta)=\langle \theta \rangle$. Suppose now that the Artin group $\A[\Gamma]$ can be decomposed as $\A[\Gamma]=H\times K$, with $H,K < \A[\Gamma]$. Then, by Lemma \ref{lemmadecompcentralizer}, we have that \[Z_{\A[\Gamma]}(\theta)=(Z_{\A[\Gamma]}(\theta)\cap H)\times (Z_{\A[\Gamma]}(\theta)\cap K)=\langle \theta \rangle\cong \mathbb{Z}.\]
We know that the infinite cyclic group $\mathbb{Z}$ is indecomposable, hence we obtain that either $Z_{\A[\Gamma]}(\theta)\cap H=\{\id\}$ or $Z_{\A[\Gamma]}(\theta)\cap K=\{\id\}$. Without loss of generality, suppose that $Z_{\A[\Gamma]}(\theta)\cap H=\{\id\}$ and that $Z_{\A[\Gamma]}(\theta)\cap K=Z_{\A[\Gamma]}(\theta)$, which means that $Z_{\A[\Gamma]}(\theta)=\langle \theta \rangle <K$. Since $[H,K]=\{\id\}$, then $H<Z_{\A[\Gamma]}(K)< Z_{\A[\Gamma]}(\theta)$, but then $H=\{\id\}$ because we assumed that $Z_{\A[\Gamma]}(\theta)\cap H=\{\id\}$. Therefore, if we write $\A[\Gamma]$ as a direct product of two subgroups $H\times K$, we find that one of these two must be trivial, meaning that the Artin group is indecomposable.
\end{proof}

\noindent With exactly the same proof, the more general following result holds.

\begin{lem}\label{lem_eltwithcycliccentralizer_Gindec}
    Let $G$ be a group, and let $g\in G$ be an element such that $Z_G(g)=\langle g\rangle$. Then $G$ is indecomposable.
\end{lem}

\subsection{Infinitely generated case}\label{infgencase}
We want now to extend Theorem \ref{SfinAindec} to some cases in which the Coxeter graph $\Gamma$ is $\infty$-connected but countable. The argument of finding the centralizer of a special element in $\A[\Gamma]$ that is the product of all the squared generators cannot be used here, because the number of generators is infinite. We suppose that the set of vertices $V(\Gamma)=S$ admits a filtration $X_0\subset X_1\subset\cdots\subset X_i\subset\cdots$ such that each $X_i$ is finite and $\Gamma_{X_i}$ is $\infty$-connected. We exploit then the subgroup $\CA[\Gamma]$ of the Artin group and the result of indecomposability of $\W[\Gamma]$ in \cite{Nui06}, when $\Gamma$ has an infinite set of vertices. \medskip\\
\noindent
Recall that given $\Gamma$ a Coxeter graph on a set of vertices $S$, the kernel of the group homomorphism $\omega:\A[\Gamma]\longrightarrow \W[\Gamma]$ sending $s\longmapsto s$ for each $s\in S$ is the colored Artin group $\CA[\Gamma]<\A[\Gamma]$.\medskip\\
\noindent The following lemma still relies on the description of the centralizer of some element in the group.
\begin{lem}\label{centralizerinCA}
    Let $\Gamma$ be a Coxeter graph on a finite set of vertices $S=\{s_1,\ldots,s_n\}$, with $\Gamma$ $\infty$-connected. Let $\theta$ be $\theta=\sigma_{s_1}^2\,\cdots \,\sigma_{s_n}^2$. Then, $\theta \in \CA[\Gamma]$ and the centralizer $Z_{\CA[\Gamma]}(\theta)$ is the infinite cyclic subgroup $\langle \theta \rangle$ generated by $\theta$.
\end{lem}
\begin{proof}
    The fact that $\theta$ actually belongs to $\CA[\Gamma]$ is trivial. By Lemma \ref{centralizerinA} we know that $Z_{\A[\Gamma]}(\theta)=\langle \theta \rangle$. Since $\CA[\Gamma]< \A[\Gamma]$, we have that $Z_{\CA[\Gamma]}(\theta)\subset Z_{\A[\Gamma]}(\theta)$, from which the result follows.
\end{proof}

\noindent Using Lemma \ref{centralizerinCA} and Lemma \ref{lemmadecompcentralizer}, as we saw for the Artin group, we get the following result.
\begin{lem}\label{SfinCAindecomposable}
    Let $\Gamma$ be an $\infty$-connected Coxeter graph on a finite set of vertices $S=\{s_1,\ldots,s_n\}$. Then, $\CA[\Gamma]$ is indecomposable.
\end{lem}
\noindent Lemma \ref{centralizerinCA} also implies the following result.
\begin{lem}\label{centerCA}
     Let $\Gamma$ be a Coxeter graph on a finite set of vertices $S=\{s_1,\ldots,s_n\}$, with $n>1$ and $\Gamma$ $\infty$-connected. Then, the center of $\CA[\Gamma]$ and its centralizer in $\A[\Gamma]$ are trivial, namely $Z(\CA[\Gamma])=Z_{\A[\Gamma]}(\CA[\Gamma])=\{\id\}$.
\end{lem}
\begin{proof}
    Define again $\theta= \sigma_{s_1}^2\,\cdots \,\sigma_{s_n}^2\in \CA[\Gamma]$. Clearly we must have \[Z_{\A[\Gamma]}(\CA[\Gamma])\subseteq Z_{\A[\Gamma]}(\theta)=\langle \theta \rangle.\]
    To complete the proof, we need to show that $\theta^l \notin Z_{\A[\Gamma]}(\CA[\Gamma])$ for any $l\in \mathbb{Z}\backslash\{0\}$.\medskip \\ 
   \noindent The infinity connection of $\Gamma$ ensures that we can express $\A[\Gamma]$ as an amalgamated product, as in the proof of Lemma \ref{centralizerinA}. For instance, set $i=1$, and let $j=j(1)\in \{2,\ldots,n\}$ be such that $m_{s_1,s_j}=\infty$. Define $\mathcal{A}_1=S\backslash \{s_j\}$, $\mathcal{B}_j=S\backslash \{s_1\}$, and $\mathcal{C}_{1j}=S\backslash \{s_1,s_j\}$. Then $\A[\Gamma]=A_1\,\frpp_{C_{1j}}\,B_j$, with the same notations as in the proof of Lemma \ref{centralhype}. Now, express $\theta=\sigma_{s_1}^2\;\cdots\,\sigma_{ s_n}^2$, as $a_1b_j$, where $a_1=\sigma_{s_1}^2\in A_1$, and $b_j=\sigma_{s_2}^2\,\cdots\, \sigma_{s_n}^2\in B_j$. By Proposition \ref{propcarina}, we know that $a_1\notin C_{1j}$ and $b_j\notin C_{1j}$. \medskip\\ 
    \noindent Now, choose two transversals $T_{A_1}$ and $T_{B_j}$ of $C_{1j}$ in $A_1$ and $B_j$, respectively, such that $a_1\in T_{A_1}$ and $b_j\in T_{B_j}$. 
    Since $a_1,b_j\notin C_{1j}$, we know that $a_1\in T_{A_1}\backslash \{\id\}$ and that $b_j\in T_{B_j}\backslash \{\id\}$. \medskip \\
   \noindent Suppose that $\theta^l\in Z_{\A[\Gamma]}(\CA[\Gamma])$, and assume for the moment that $l\geq 1$. Then $\theta^l$ must commute with $\sigma_{s_1}^2=a_1\in \CA[\Gamma]$. Namely, \begin{equation}\label{commutaconsigmauno}
   \sigma_{s_1}^2 \,\theta^l = \theta^l\, \sigma_{s_1}^2 \qquad \Longleftrightarrow \qquad a_1 \,\theta^l = \theta^l\, a_1 = a_1b_j \,\theta^{l-1}\,a_1\qquad \Longleftrightarrow \qquad \theta^l =b_j\, \theta^{l-1}\, a_1.
    \end{equation}
    \noindent This leads to the expression
    \[
    \underbrace{a_1\, b_j\, a_1\,\cdots b_j}_{2l \mbox{\;\small{elements}}}=
    \underbrace{ b_j\, a_1\,b_j\,\cdots a_1}_{2l \mbox{\;\small{elements}}}.
    \]
    \noindent where both sides represent expressions of the element $\theta^l\sigma_{s_1}^2$ as strictly alternating $(T_{A_1}\backslash \{\id\})\cup (T_{B_j}
        \backslash\{\id\})$-product. By the uniqueness of the normal form in an amalgamated product (see Theorem \ref{normalformamalgam} and Lemma \ref{alternformamalg}), we obtain that Equation \ref{commutaconsigmauno} cannot hold for any $l\geq 1$. If $l<0$, $\theta^l$ cannot commute with $\sigma_{s_1}^2$, since otherwise $\theta^{-l}$ would also commute with $\sigma_{s_1}^2$, contradicting the previous result.\medskip \\
        \noindent Therefore, we conclude that $Z(\CA[\Gamma])=Z_{\A[\Gamma]}(\CA[\Gamma])=\{\id\}$.
\end{proof}

\noindent The fact that the colored Artin group is indecomposable when $\Gamma$ is $\infty$-connected for a finite graph, will allow us to generalize the indecomposability also for the infinitely generated case. 
\noindent
Now, suppose that $\Gamma$ is a $\infty$-connected Coxeter graph on an infinite countable set $S$, and assume that $S$ admits a filtration $X_0\subset X_1\subset\cdots\subset X_i\subset\cdots$ such that each $X_i$ is finite and each induced subgraph $\Gamma_{X_i}$ is $\infty$-connected. Consequently, $\Gamma$ can be expressed as the union $\Gamma=\bigcup_{i=1}^{\infty}\Gamma_i$ of the subgraphs $\Gamma_i:=\Gamma_{X_i}$, and $S=\bigcup_{i=1}^{\infty}X_i$. We recall that 
\begin{align*}
     \A[\Gamma]\;\;=\, \underset{X_i}{\varinjlim }\,\,\A[\Gamma_{X_i}]; \qquad \text{ and }\qquad \CA[\Gamma]\;\;=\, \underset{X_i}{\varinjlim }\,\,\CA[\Gamma_{X_i}];.
\end{align*}
In this setting we can extend the result of indecomposability to the infinitely generated case.

\begin{lem}\label{SinfCAindec}
    Let $\Gamma$ be an infinite Coxeter graph with vertex set $S$ that admits a filtration $X_0\subset X_1\subset\cdots\subset X_i\subset\cdots$ such that $X_i$ is a finite set for all $i<\infty$ and $S=V(\Gamma)=\bigcup_{i=1}^{\infty}X_i$. Suppose that $\Gamma$ and $\Gamma_i:=\Gamma_{X_i}$ are $\infty$-connected for all $i<\infty$. Then, the colored Artin group $\CA[\Gamma]$ is indecomposable.
\end{lem}
\begin{proof}
    Take any $i<\infty$, and consider the finite, $\infty$-connected Coxeter graph $\Gamma_i$ on the finite set $X_i$. Take now the group homomorphism induced by the construction of Godelle and Paris (see Subsection \ref{parabolic}), $p_{X_i}:\CA[\Gamma]\longrightarrow \CA[\Gamma_i]$. As we said in Remark \ref{godelleparisforinfinity}, this group homomorphism is well defined even if $S$ is infinite, and its restriction to $\CA[\Gamma_i]$ is the identity. \medskip
    \\
    \noindent Suppose that the colored Artin group can be written as a direct product of subgroups as $\CA[\Gamma]=H\times K$. We want to prove that one of these two direct factors is trivial. Suppose that neither one of the two is trivial, so there exist two elements $h\in H$ and $k\in K$, with $h\neq \id \neq k$. Take $g=hk\in \CA[\Gamma]$. Since $\CA[\Gamma]\,=\;{\varinjlim}_{i\in \mathbb{N}}\;\;\CA[\Gamma_i]$, we can always choose some $i<\infty$ such that $h,k \in \CA[\Gamma_i]<\CA[\Gamma]$, where $X_i$ is a finite set. Thus, $g=hk$ also belongs to $\CA[\Gamma_i]$. Consider now the homomorphism
    \[p_{X_i}: \CA[\Gamma]=H\times K\longrightarrow \CA[\Gamma_i].\]
    Since it respects the commutation, we have that $p_{X_i}(H)$ must commute with $p_{X_i}(K)$. The intersection of these two groups must commute with both $p_{X_i}(H)$ and $p_{X_i}(K)$, then, it must be contained in the center of $\CA[\Gamma_i]$. By Lemma \ref{centerCA}, we know that, since $\Gamma_i$ is finite and $\infty$-connected, $Z(\CA[\Gamma_i])=\{\id\}$. Thus, $p_{X_i}(H)\cap p_{X_i}(K)=\{\id\}$ and we have a decomposition of $\CA[\Gamma_i]$ as $p_{X_i}(H)\times p_{X_i}(K)$. But, by Lemma \ref{SfinCAindecomposable}, we have that $\CA[\Gamma_i]$ is indecomposable, so it must be either $p_{X_i}(H)=\{\id\}$ or $p_{X_i}(K)=\{\id\}$. But we supposed that both $h,k$ belonged to $\CA[\Gamma_i]$, and since $p_{X_{i}}$ is the identity restricted to $\CA[\Gamma_i]$, we find out that $p_{X_i}(h)=h\neq \id$ and $p_{X_i}(k)=k\neq \id$. We have then a contradiction, which means that either $H$ or $K$ is trivial and that $\CA[\Gamma]$ is indecomposable.
\end{proof}
\noindent
This lemma will be a key step to prove that indeed the Artin group $\A[\Gamma]$ is indecomposable. Before doing so, we show a preliminary result.

\begin{lem}
    Let $\Gamma$ be as in Lemma \ref{SinfCAindec}. Then, the centralizer $Z_{\A[\Gamma]}(\CA[\Gamma])=\{\id\}$.
\end{lem}
\begin{proof}
    Take an element $g\in Z_{\A[\Gamma]}(\CA[\Gamma])$. Since $\A[\Gamma]\,=\;{\varinjlim}_{i\in \mathbb{N}}\;\;\A[\Gamma_i]$, we can find some $i<\infty$ such $g\in \A[\Gamma_i]$. If $g$ commutes with all the elements of $\CA[\Gamma]$, then it must commute in particular with the elements in $\CA[\Gamma_i]$. Since $\Gamma_i$ is finite and $\infty$-connected, we can apply Lemma \ref{centerCA} and we obtain that $g\in Z_{\A[\Gamma_i]}(\CA[\Gamma_i])=\{\id\}$. Thus, $Z_{\A[\Gamma]}(\CA[\Gamma])$ cannot contain nontrivial elements, and we get $Z_{\A[\Gamma]}(\CA[\Gamma])=\{\id\}$.
\end{proof}
\noindent 
The indecomposability for Coxeter groups has been studied by Paris in \cite{Par07}. In this article, the author shows the following result for $\Gamma$ a finite Coxeter graph:
\begin{thm}[Theorem 4.1 in \cite{Par07}]\label{decompW}
Let $\Gamma$ be a connected Coxeter graph that is not of spherical type. Then the Coxeter group $\W[\Gamma]$ is indecomposable.
\end{thm}
\noindent Moreover, by the same article, we know that if $\Gamma$ is of spherical type and its center is not trivial, then $\W[\Gamma]$ admits a well-described decomposition. However, in the cited work, the proof of indecomposability for $\W[\Gamma]$ when $\Gamma$ is non-spherical deeply relies on the finiteness of the generating set. Specifically, it uses that the Coxeter element (which is the product of all the generators in some order) is \textit{essential}, which means that it does not belong to any proper parabolic subgroup. Since for an infinite number of generators the Coxeter element is not well defined, the proof shown of Theorem 4.1 in \cite{Par07} cannot be readapted to the infinite case. The result is true anyway, and it has been shown by Nuida in \cite{Nui06}.

\begin{thm}[Theorem 3.3 in \cite{Nui06}]\label{Wnonsphindec}
Let $\Gamma$ be a connected Coxeter graph that is not of spherical type, of arbitrary rank. Then the Coxeter group $\W[\Gamma]$ is indecomposable.
\end{thm}
\noindent We can now get to the indecomposability result for Artin groups in the infinitely generated case.

\begin{thm}\label{SinfAindec}
 Let $\Gamma$ be an infinite Coxeter graph with vertex set $S$ that admits a filtration $X_0\subset X_1\subset\cdots\subset X_i\subset\cdots$ such that $X_i$ is a finite set for all $i<\infty$ and $S=V(\Gamma)=\bigcup_{i=1}^{\infty}X_i$. Suppose that $\Gamma$ and $\Gamma_i:=\Gamma_{X_i}$ are $\infty$-connected for all $i<\infty$. Then the Artin group $\A[\Gamma]$ is indecomposable.
\end{thm}
\begin{proof}
    Suppose that $\A[\Gamma]$ is decomposable as $\A[\Gamma]=H\times K$ with $H$ and $K$ normal subgroups. We want to show that one of these two direct factors is trivial. We consider the group homomorphism $\omega:\A[\Gamma]\longrightarrow \W[\Gamma]$, and we analyze the images $\omega(H)$ and $\omega(K)$ of the factors. Since $\omega$ is a group homomorphism, they must commute. Their intersection $\omega(H)\cap\omega(K)$ must then commute with both, and so it must be contained in the center $Z(\W[\Gamma])$. By known results about the theory of Coxeter groups (see \cite[Exercice 3, page 130]{Bourbaki}), we know that $Z(\W[\Gamma])=\{\id\}$. Thus $\omega(H)\times \omega(K)$ is a direct decomposition of $\W[\Gamma]$.\medskip\\
    \noindent By Theorem \ref{Wnonsphindec}, we get that $\W[\Gamma]$ is indecomposable, so either $\omega(H)$ or $\omega(K)$ is trivial. Without loss of generality, suppose that $\omega(K)=\{\id\}$ and that $\omega(H)=\W[\Gamma]$, which implies that $K< \CA[\Gamma]$. Take an element $g\in \CA[\Gamma]$. Then it can be written as $g=hk$, with $h\in H$ and $k\in K<\CA[\Gamma]$. But then $h=gk^{-1}\in \CA[\Gamma]$, which means that we can decompose $\CA[\Gamma]$ as $(\CA[\Gamma]\cap K)\times (\CA[\Gamma]\cap H)$. Since $K<\CA[\Gamma]$, we rewrite $\CA[\Gamma]=K\times (\CA[\Gamma]\cap H)$. In Lemma \ref{SinfCAindec} we found that $\CA[\Gamma]$ is indecomposable, which means that either that $K=\{\id\}$, in which case we are done, or that $\CA[\Gamma]\cap H=\{\id\}$. In this last case, we obtain that $\CA[\Gamma]=K$, so $\A[\Gamma]=\CA[\Gamma]\times H$. This decomposition requires $H$ to commute with $K=\CA[\Gamma]$, so $H\subset Z_{\A[\Gamma]}(\CA[\Gamma])$. From Lemma \ref{centerCA}, we know that $Z_{\A[\Gamma]}(\CA[\Gamma])=\{\id\}$, therefore $H=\{\id\}$ and $\A[\Gamma]$ is indecomposable.
\end{proof}

\section{Indecomposability of $\KVA[\Gamma]$}\label{kvaindec}
\noindent In Section \ref{fingencase}, we showed that if $\Gamma$ is a $\infty$-connected finite Coxeter graph, then the Artin group $\A[\Gamma]$ is indecomposable. Then, in Section \ref{infgencase}, we generalized the relation between the property of $\infty$-connection of $\Gamma$ and the indecomposability of $\A[\Gamma]$ also to some infinitely generated cases. Specifically, we found that $\A[\Gamma]$ is indecomposable as a direct product of nontrivial subgroups if $\Gamma$ is an infinite and $\infty$-connected Coxeter graph on a set $S$, with $S=\bigcup_{i=1}^\infty X_i$ and $X_i\subset X_j$ if $i<j$. We also requested that $\Gamma_i=\Gamma_{X_i}$ is finite and $\infty$-connected.\medskip
\\
\noindent In this section, we want to apply these results to the Artin group $\A[\hGamma]=\KVA[\Gamma]$. We recalled the definition of $\hGamma$ at the end of Subsection \ref{subs-virtart}. First, in Corollary \ref{kvaindecspherical}, we will treat the case in which $\hGamma$ is finitely generated, namely, when $\Gamma$ is of spherical type. Then, in Theorem \ref{kvaindecnonspherical}, we will use the generalized result of indecomposability for infinitely generated graphs, and apply it to the case in which $\Gamma$ is not of spherical type and $\hGamma$ is not finite. Of course, to be able to apply these results to $\hGamma$, we must check that $\hGamma$ satisfies the hypotheses of being $\infty$-connected and the direct limit of finite subgraphs $\hGamma_{\cX_k}$, each $\infty$-connected and finite.\medskip
\\
\noindent First, we show that $\hGamma$ is connected when $\Gamma$ is connected. Therefore, we will deduce that for any connected $\Gamma$, the Coxeter graph $\hGamma$ is always $\infty$-connected. 

\begin{prop}\label{gammahatconnected}
    Let $\Gamma$ be a Coxeter graph with set of vertices $S$, and let $\hGamma$ be the associated Coxeter root graph. If $\Gamma$ is connected, then $\hGamma$ is connected.
\end{prop}
\begin{proof}
Take any two vertices $\beta,\gamma$ of $\hGamma$, namely, any two roots $\beta$ and $\gamma$ in $\Phi[\Gamma]$, and we want to prove that there exists a path in $\hGamma$ connecting them. Suppose that there exists an element $w\in \W[\Gamma]$ and two generators $s,t\in S$ such that $\beta=w(\alpha_s)$ and $\gamma=w(\alpha_t)$. We have then, by definition of $\hGamma$, that $\hmbg=m_{s,t}$. If $3\leq\hmbg\leq \infty$, then the two vertices $\beta$ and $\gamma$ are joined by an edge labeled by $\hmbg$ (or unlabeled if $\hmbg=3$), so they are connected. If instead $\hmbg=m_{s,t}=2$, there is no edge between the vertices $\beta$ and $\gamma$, but we show that there is a path in $\hGamma$ joining them.
Since $\Gamma$ is connected, for every couple of vertices $s,t$ in $S$, there is a path $r_1, \ldots, r_n$ in $\Gamma$ such that $r_1=s$, $r_n=t$ and for all $i\in \{1,\ldots,n-1\}$ we have $m_{r_i,r_{i+1}}\geq 3$. Set now $\delta_i=w(\alpha_{r_i})$ for every $1\leq i\leq n$. We obtain that $\delta_1=w(\alpha_s)=\beta$ and that $\delta_{n}=w(\alpha_t)=\gamma$. Furthermore, we have that $\hm_{\delta_i,\delta_{i+1}}=m_{r_i,r_{i+1}}\geq 3$ for all $i\in \{1,\ldots,n-1\}$, so each $\delta_i$ is adjacent to $\delta_{i+1}$ and the sequence $\delta_1,\ldots,\delta_n$ is a path in $\hGamma$ joining $\beta$ and $\gamma$.\\
By the definition of $\hGamma$, if such $w\in \W[\Gamma]$ and $s,t\in S$ do not exist, we have that $\hmbg=\infty$ and $\beta$ and $\gamma$ are adjacent, thus, connected.
\end{proof}
\noindent In particular, we showed that if $\W[\Gamma]$ is irreducible, then $\KVA[\Gamma]$, as an Artin group, is irreducible. We now move to study the infinity-connection of $\hGamma$.
\begin{prop}\label{gammahatinftyconnected}
    Let $\Gamma$ be a connected Coxeter graph, and let $\hGamma$ be the associated Coxeter root graph. Then $\hGamma$ is $\infty$-connected.
\end{prop}
\begin{proof}
    Consider the subgraph $\hGamma^{\infty}$ of $\hGamma$, where all the edges are removed except those labeled by $\infty$. Taken any two roots $\beta,\gamma$ in $\Phi[\Gamma]$, we want to prove that there exists a path in $\hGamma^{\infty}$ connecting them. Suppose on a first place that $\beta$ and $\gamma$ are joined with an edge in $\hGamma$, that is, $\hm_{\beta,\gamma}\geq 3$. Suppose now that $\gamma$ belongs to a connected component of $\hGamma^{\infty}$ different from the one of $\beta$. Specifically, there is no infinity-labeled edge between $\beta$ and $\gamma$, hence $3\leq\hmbg<\infty$. In particular, $\langle \beta,\gamma\rangle=-\cos\left(\frac{\pi}{\hmbg}\right)=-\langle -\beta,\gamma\rangle\leq 0$. Recall that a root $\beta$ is always connected to its opposite $-\beta$ by an $\infty$-labeled edge. Since there is no path in $\hGamma^{\infty}$ between $\gamma$ and $\beta$, there must be no $\infty$-labeled edge joining $\gamma$ and $-\beta$, otherwise we would have a contradiction. Hence, $2\leq\hm_{-\beta,\gamma}<\infty$, that means $-\langle -\beta,\gamma\rangle=+\cos\left(\frac{\pi}{\hm_{-\beta,\gamma}}\right)\geq0$. But then, the only possibility is that $\hmbg=\hm_{-\beta,\gamma}=2$, which is a contradiction since we assumed $\hm_{\beta,\gamma}\geq 3$.\medskip \\
    \noindent Suppose now that $\beta$ and $\gamma$ are not adjacent, namely that $\hmbg=2$. By Proposition \ref{gammahatconnected}, we know that $\hGamma$ is connected, and there exists a path in $\hGamma$ with vertices $\delta_1,\ldots,\delta_n$ such that $\delta_1=\beta$, $\delta_n=\gamma$ and $\hm_{\delta_i,\delta_{i+1}}\geq 3$ for all $i\in \{1,\ldots,n-1\}$. Applying the argument in the previous paragraph to each pair of adjacent vertices $\delta_i$ and $\delta_{i+1}$, we find that there must be an $\infty$-labeled path in $\hGamma$ joining $\delta_i$ and $\delta_{i+1}$. Concatenating all such paths, we find an $\infty$-labeled path joining $\beta$ with $\gamma$, and therefore we can conclude.
    \end{proof}
\noindent As in the previous sections, if not differently specified, every Coxeter graph $\Gamma$ that we will consider during this section will be connected. 
    \begin{cor}\label{kvaindecspherical}
        If $\Gamma$ is a Coxeter graph of spherical type, then the group $\KVA[\Gamma]$ is indecomposable.
    \end{cor}
\begin{proof}
    This is a straightforward consequence of the previous results. Knowing that $\KVA[\Gamma]=\A[\hGamma]$, when $\Gamma$ is of spherical type, implies that $\A[\hGamma]$ is finitely generated, and by Proposition \ref{gammahatinftyconnected} we know that $\hGamma$ is $\infty$-connected. Then, by Theorem \ref{SfinAindec}, we get that $\KVA[\Gamma]$ is indecomposable.
\end{proof}
\noindent We analyze now the case in which $\Gamma$ is not of spherical type. By Propositions \ref{gammahatconnected} and \ref{gammahatinftyconnected}, we know that if $\Gamma$ is connected, then $\hGamma$ is $\infty$-connected. By \cite{deodh82}, we know that the root system $\Phi[\Gamma]=\{w(\alpha_s)\,|\,w\in \W[\Gamma], s\in S\}$ can be partitioned into positive roots and negative roots, namely, $\Phi[\Gamma]=\Phi[\Gamma]^+\sqcup (-\Phi[\Gamma]^+)$. Set now a total ordering on $\Phi[\Gamma]^+$, so that $\Phi[\Gamma]^+=\{\beta_1,\beta_2,\ldots,\beta_i,\ldots\}$ and $\Phi[\Gamma]^-=-\Phi[\Gamma]^+=\{-\beta_1,-\beta_2,\ldots\}$. Now, for each $i\in \{1,2,\ldots\}$, let $\cX_i$ be $\cX_i^+\sqcup \cX_i^-\subset\Phi[\Gamma]=V(\hGamma)$, where
\begin{align*}
    \cX_i^+:=\{\beta_1, \beta_2,\ldots,\beta_i\},\qquad \qquad \cX_i^-=-\cX
    _i^+=\{-\beta_1, -\beta_2,\ldots,-\beta_i\}.
\end{align*}
\noindent For instance, $\cX_1=\{\beta_1,-\beta_1\}$, and $\cX_2=\{\beta_1,\beta_2,-\beta_1,-\beta_2\}$. By definition of $\hGamma$, we know that for each $\beta\in \Phi[\Gamma]$, the vertices in $\hGamma$ represented by $\beta$ and $-\beta$ are always joined by an $\infty$-labeled edge. We show now that it is possible to set a total ordering on $\Phi[\Gamma]^+$ in order to have that the finite subgraph $\hGamma_{\cX_i}$ is connected for every $i\geq 1$.
\begin{lem}\label{partialorderconnected}
    Let $\Gamma$ be a Coxeter graph with $\hGamma$ the Coxeter graph associated with the root system $\Phi[\Gamma]$. Then there exists a total ordering on $\Phi[\Gamma]^+=\{\beta_1,\beta_2,\ldots\}$ such that for each $\cX_i$ as defined above, $\hGamma_{\cX_i}$ is connected.
\end{lem}
\begin{proof}
    We construct $\cX_i$ by induction on $i$.
    Pick any root $\beta$ in $\Phi[\Gamma]^+$, and set $\beta=\beta_1$. The Coxeter subgraph $\hGamma_{\cX_1}$ is clearly connected. Since $\Gamma$ is a connected Coxeter graph, so it is $\hGamma$ by Proposition \ref{gammahatconnected}, thus, there exists at least one root $\gamma\in \Phi[\Gamma]$ such that $\hm_{\beta_1,\gamma}\neq 2$. If $\gamma\in \Phi[\Gamma]^+$, we then set $\gamma=\beta_2$. Otherwise, we also have that $\hm_{\beta_1,-\gamma}\neq 2$, so we set $\beta_2=-\gamma\in \Phi[\Gamma]^+$. In both cases, we get that $\cX_2=\{\beta_1,\beta_2,-\beta_1,-\beta_2\}$, so $\hGamma_{\cX_2}$ is connected.\medskip
    \\
    \noindent Suppose now by induction that $\Gamma_{\cX_i}$ is connected for some $i\geq 2$. Since $\hGamma$ is connected, there exist roots $\beta\in \cX_i$ and $\delta\in \Phi[\Gamma]\backslash\cX_i$ such that $\beta$ and $\delta$ are adjacent. Up to replacing $\beta$ by $-\beta$ and $\delta$ by $-\delta$, we may assume that $\beta\in \cX_i^+$, say $\beta=\beta_j$ for some $j\leq i$. If $\delta$ is a positive root, then set $\beta_{i+1}=\delta$. Otherwise, since $\hm_{\beta_j,\delta}\neq 2$ implies $\hm_{\beta_j,-\delta}\neq 2$, we set then $\beta_{i+1}=-\delta$. In both cases, by construction, we get that $\hGamma_{\cX_{i+1}}$ is connected.\\
    
  Therefore, the total ordering that we set on $\Phi[\Gamma]$ going on in this fashion is such that $\hGamma_{\cX_i}$ is connected for each $i\geq 1$.
\end{proof}

\noindent We show in the next lemma that the finite subgraphs $\hGamma_{\cX_i}$ that we just defined are actually $\infty$-connected.
\begin{lem}\label{cXiinftyconnected}
    Let $\Gamma$ be a Coxeter graph, with $\hGamma$ the Coxeter graph associated with the root system $\Phi[\Gamma]$. Let $\cX_i$ be the finite subsets of $\Phi[\Gamma]=V(\hGamma)$ as above, such that $\hGamma_{\cX_i}$ is connected for every $i\geq 1$. Then, for each $i$, $\hGamma_{\cX_i}$ is $\infty$-connected.
\end{lem}
\begin{proof}
    Again, we proceed by induction on $k$. Of course, since $\cX_1=\{\beta_1,-\beta_1\}$ and $\hm_{\beta_1,-\beta_1}=\infty$, $\hGamma_{\cX_1}$ is $\infty$-connected. Let us analyze now $\cX_2=\{\beta_1,\beta_2,-\beta_1,-\beta_2\}$. If $\hm_{\beta_1,\beta_2}=\infty$, we get immediately that $\hGamma_{\cX_2}$ is $\infty$-connected, so let us suppose that $\hm_{\beta_1,\beta_2}\neq \infty$. We look now at the coefficient $\hm_{-\beta_1,\beta_2}$. If it equals $\infty$, then all the four vertices in $\hGamma_{\cX_2}$ are joined by some path in $\hGamma_{\cX_2}^{\infty}$, so the graph is $\infty$-connected. Suppose then that $\hm_{\beta_1,\beta_2}\neq \infty \neq \hm_{-\beta_1,\beta_2}$. As in the proof of Proposition \ref{gammahatinftyconnected}, we obtain that
    \[
    0\geq -\cos{\frac{\pi}{\hm_{\beta_1,\beta_2}}}=\langle \beta_1,\beta_2\rangle=-\langle- \beta_1,\beta_2\rangle=+\cos{\frac{\pi}{\hm_{-\beta_1,\beta_2}}}\geq 0;
    \]
    from which we deduce that $\hm_{\beta_1,\beta_2}=\hm_{-\beta_1,\beta_2}=2$. But this would mean that $\beta_2$ belongs to a different connected component from that of $\beta_1,-\beta_1$, which is absurd since $\hGamma_{\cX_2}$ is connected. Therefore, at least one between $\beta_1$ and $-\beta_1$ is connected to $\beta_2$ through an $\infty$-labeled edge, so $\hGamma_{\cX_2}$ is $\infty$-connected. \medskip\\
    \noindent Suppose now by induction that $\hGamma_{\cX_i}$ is $\infty$-connected. By the construction made in the proof of Lemma \ref{partialorderconnected}, $\cX_{i+1}^+=\cX_i^+\sqcup \{\beta_{i+1}\}$, where $\beta_{i+1}$ is a positive root such that there exists at least one $\beta_j$ in $\cX_i^+$ (with $j\leq i$) such that $\hm_{\beta_j,\beta_{i+1}}\neq 2$. If $\hm_{\beta_j,\beta_{i+1}}=\infty$ or $\hm_{-\beta_j,\beta_{i+1}}=\infty$, thanks to the induction hypotheses we obtain that $\beta_{i+1}$ is connected to each root in $\cX_i$ through an $\infty$-labeled path, and so that $\hGamma_{\cX_{i+1}}$ is $\infty$-connected. If both $\hm_{\beta_j,\beta_{i+1}}$ and $\hm_{-\beta_j,\beta_{i+1}}$ are different from infinity, as in the previous argument, we get that $\hm_{\beta_j,\beta_{i+1}}=2$, which is absurd. Hence, at least one of the two must be infinity, and $\hGamma_{\cX_{i+1}}$ is $\infty$-connected. By induction, $\hGamma_{\cX_i}$ is $\infty$-connected for each $i\geq1$.
\end{proof}
\begin{rmk}\label{unioncXradici}
    Let $\Gamma$ be a Coxeter graph and $\Phi[\Gamma]$ be its root system. Fix a total ordering on $\Phi[\Gamma]^+$ such that for each $i\geq 1$ the subsets $\cX_i$ give rise to the connected Coxeter graphs $\hGamma_{\cX_i}$. By construction, we have that each set $\cX_i$ is finite, that $\cX_i\subset \cX_j$ if $i<j$, and $\Phi[\Gamma]=V(\hGamma)=\bigcup_{i=1}^\infty\cX_{i}$. Moreover, thanks to Lemma \ref{cXiinftyconnected}, we know that each finite subgraph $\hGamma_{\cX_i}$ is $\infty$-connected.
\end{rmk}
\noindent We are now ready to prove the indecomposability of $\KVA[\Gamma]$ for $\Gamma$ not of spherical type.

\begin{thm}\label{kvaindecnonspherical}
    Let $\Gamma$ be a Coxeter graph, non necessarily of spherical type. Then the group $\KVA[\Gamma]$ is indecomposable.
\end{thm}
\begin{proof}
    Recall that $\KVA[\Gamma]=\A[\hGamma]$ and that $V(\hGamma)=\Phi[\Gamma]$, which is a non-necessarily finite set of vertices. If $\Gamma$ is of spherical type, by Corollary \ref{kvaindecspherical} we can conclude. Suppose then that $\Gamma$ is not of spherical type.\\
    We saw that, if we have a total ordering on $\Phi[\Gamma]^+=\{\beta_1,\beta_2,\beta_3,\ldots\}$, we can define for every $i\geq1$ a set $\cX_i=\cX_i^+\sqcup \cX_i^-$, where $\cX_i^+=\{\beta_1,\ldots,\beta_i\}$. We fix now a total ordering on $\Phi[\Gamma]^+=\{\beta_1,\beta_2,\beta_3,\ldots\}$ such that for all $i\geq 1$, the subgraph $\hGamma_{\cX_i}$ is connected. In Lemma \ref{partialorderconnected} we showed that such an order always exists, and in Lemma \ref{cXiinftyconnected} we showed that $\hGamma_{\cX_i}$ is also $\infty$-connected for each $i\geq 1$. In Remark \ref{unioncXradici}, we observed that $\Phi[\Gamma]=V(\hGamma)=\bigcup_{i=1}^\infty\cX_{i}$. We can now apply Theorem \ref{SinfAindec} to the infinite Coxeter graph $\hGamma$ whose vertices are $\Phi[\Gamma]=V(\hGamma)=\bigcup_{i=1}^\infty\cX_{i}$, where each $\hGamma_{\cX_i}$ is finite and $\infty$-connected. Thus, $\A[\hGamma]=\;\varinjlim_{i\in \mathbb{N}}\;\A[\hGamma_{\cX_i}]=\KVA[\Gamma]$ is indecomposable.
\end{proof}

\section{Centralizer of $\KVA[\Gamma]$}\label{sectioncentralizerofKVA}
We recall that in \cite{BellParThiel} the authors prove many results concerning the center and some group centralizers in the virtual Artin group $\VA[\Gamma]$. In particular, they show that for any Coxeter graph $\Gamma$, the center $Z(\VA[\Gamma])$ is trivial. Furthermore, they prove that the centralizer of the Coxeter group $\W[\Gamma]$, viewed as a subgroup of $\VA[\Gamma]$, is equal to its center, so $Z_{\VA[\Gamma]}(\W[\Gamma])=Z(\W[\Gamma])$, (which is nontrivial only when $\Gamma$ is of spherical type and its longest element is central). Finally, they show that, for any $\Gamma$, the center of $\KVA[\Gamma]=\A[\hGamma]$ is trivial.\\
\\
In this section we study the centralizer of $\KVA[\Gamma]$ in $\VA[\Gamma]$. To do so, we recall that, as an Artin group, $\A[\hGamma]=\KVA[\Gamma]$ can never be of spherical type. Even when $\hGamma$ has a finite number of vertices, we showed in Proposition \ref{gammahatinftyconnected} that it is $\infty$-connected, and so it gives rise to an infinite Coxeter group. Moreover, we summarized in Remark \ref{unioncXradici} that $V(\hGamma)=\Phi[\Gamma]$ can be written as the (possibly) infinite union \[\Phi[\Gamma]\;=\;\bigcup_{i=1}^{\infty} \cX_i,\] with $\cX_i\subset \cX_j$ if $i<j$, with $\cX_i\subset \Phi[\Gamma]$ finite and such that $\hGamma_{\cX_i}$ is $\infty$-connected. Observe that the Coxeter group associated to the $\infty$-connected Coxeter graph $\hGamma_{\cX_i}$ is infinite for every $i\in \mathbb{N}$. These facts will be used to show some important properties concerning the quasi-center of the Coxeter group $\W[\hGamma]$.\\
\\
\noindent Let $\Gamma$ be a Coxeter graph on a set of vertices $S$. The \textit{quasi-center} $QZ(\W[\Gamma])$ is \[\{w\in \W[\Gamma]\;|\;  wSw^{-1}=S\}.\] The quasi-center of an Artin group is defined in an analogous way. When $S$ is a finite set, there is the following result.

\begin{thm}[Exercise 3, page 130 in \cite{Bourbaki}]\label{qzBourb}
Let $(\W[\Gamma],S)$ be a Coxeter system with $S$ finite. Let $w\in\W[\Gamma]$ be such that $wSw^{-1}=S$. Then, such $w$ is non-trivial only when $\W[\Gamma]$ is finite, in which case it coincides with the longest element $w=w_0$ in $\W[\Gamma]$.
\end{thm}
\noindent The previous theorem shows in particular that, for $S$ finite, the quasi-center $QZ(\W[\Gamma])$ is $\{1,w_0\}$ when $\Gamma$ is of spherical type (so $w_0$ exists), and it is trivial otherwise. We recalled the definition and some properties of the longest element in Subsection \ref{parabolic} and Proposition \ref{proplongestelement}. \medskip
\\
\noindent
We aim now to generalize Theorem \ref{qzBourb} to an infinitely generated Coxeter group.

\begin{thm}\label{Qzinfingen}
Let $\Gamma$ be a connected Coxeter graph on an infinite countable set $S$. Then its quasi-center $QZ(\W[\Gamma])$ is trivial. 
\end{thm}
\begin{proof}
     Suppose that there exists an element $w\in QZ(\W[\Gamma])$; we want to show that necessarily $w=\id$. Even though $S$ is infinite, the support of each element is finite, so we choose $X\subset S$ such that $supp(w)\subset X$ and $|X|<\infty$. We can suppose without loss of generality that $X$ is connected. Since $X=S\cap \W[\Gamma_X]$, $w\W[\Gamma_{X}]w^{-1}\subset \W[\Gamma_X]$, and $wSw^{-1}=S$, it follows that $wXw^{-1}\subset X$. Since $X$ is finite, we have $wXw^{-1}=X$, implying that $w\in QZ(\W[\Gamma_X])$. If $\Gamma_X$ is not of spherical type, then by Theorem \ref{qzBourb} we conclude $w=\id$. \medskip \\ \noindent Suppose now that $\Gamma_X$ is of spherical type. We have that $w\in QZ(\W[\Gamma_X])=\{\id, w_0^X\}$, where $w_0^X$ is the longest element in the finite group $\W[\Gamma_X]$. Since $S$ is infinite and $\Gamma$ is connected, there exists $Y\subset S$ such that $X$ is strictly contained in $Y$, with $Y$ finite and $\Gamma_Y$ connected. In particular, $w\in \W[\Gamma_X]<\W[\Gamma_Y]$, which implies that \[wYw^{-1}\subset\W[\Gamma_Y]\cap S=Y.\]
     This shows that $w\in QZ(\W[\Gamma_Y])$.
     Therefore, 
     \[w\in QZ(\W[\Gamma_X])\cap QZ(\W[\Gamma_Y])=\{\id,w_0^X\}\cap QZ(\W[\Gamma_Y]).\]
     We want now to show that $w_0^X\notin QZ(\W[\Gamma_Y])$. If $\Gamma_Y$ is not of spherical type, then $QZ(\W[\Gamma_Y])=\{\id\}$ and the result is clear. Otherwise, $QZ(\W[\Gamma_Y])=\{\id, w_0^Y\}$. But $w_0^X\neq w_0^Y$. Indeed, all the generators of a spherical type Coxeter group must appear at least once in any (reduced) expression of its longest element (see Subsection \ref{parabolic} or \cite{deodh82}). Knowing that the support of an element does not depend on the choice of its reduced expression, we can state that $supp(w_0^X)=X$ and $supp(w_0^Y)=Y$. Since $X$ is strictly contained in $Y$, the supports of $w_0^X$ and $w_0^Y$ are different. Therefore, $\{\id,w_0^X\}\cap \{\id,w_0^Y\}=\{\id\}$, and $w=\id$.
\end{proof}

\noindent The main application of the previous theorem occurs when the Coxeter graph is $\hGamma$, which is always $\infty$-connected when $\Gamma$ is connected, it is never of spherical type, and has infinitely many vertices whenever $\Gamma$ is of non-spherical type.

\begin{cor}\label{QZWhGamma}
    Let $\Gamma$ be a connected Coxeter graph and let $\hGamma$ be the associated Coxeter root graph. Then $QZ(\W[\hGamma])=\{\id\}$.
\end{cor}
\noindent The proof of this corollary follows straightforward from Proposition \ref{gammahatconnected} and Theorem \ref{Qzinfingen}. We now have the necessary tools to show the following.
\begin{thm}\label{centralizerKVA}
Let $\Gamma$ be any connected Coxeter graph. Then $Z_{\VA[\Gamma]}(\KVA[\Gamma])=\{\id\}$.
\end{thm}
\begin{proof}
    Recall that $\KVA[\Gamma]=\A[\hGamma]$ and that there exists a group homomorphism
    \[
    \omega:\A[\hGamma]\longrightarrow \W[\hGamma].
    \]
    \noindent Take now $g\in Z_{\VA[\Gamma]}(\KVA[\Gamma])$, and call $\widehat{\mathcal{S}}=\{\delta_{\beta}\;|\;\beta\in \Phi[\Gamma]\}$ the generating set of $\KVA[\Gamma]$ as an Artin group, and denote by $\widehat{S}=\{\widehat{s}_{\beta}\,|\, \beta\in \Phi[\Gamma]\}$ the generating set of the corresponding Coxeter group $\W[\hGamma]$. With these notations, $\omega(\delta_{\beta})=\widehat{s}_{\beta}\in \widehat{S}$ for all $\delta_{\beta}\in \widehat{\mathcal{S}}$. If $g$ belongs to the centralizer of $\KVA[\Gamma]$, we have then that $g\,\delta_{\beta}=\delta_{\beta}\,g$ for all $\beta\in \Phi[\Gamma]$. In particular, since $g\in \VA[\Gamma]=\KVA[\Gamma]\rtimes \W[\Gamma]$, we can write it as $g=k\cdot w$, with $k\in \KVA[\Gamma]$ and $w\in\W[\Gamma]$. Rewriting the previous equality, we get that $k\,w\,\delta_{\beta}=\delta_{\beta}\,k\,w$, which implies
    \[\,w\,\delta_{\beta}\,w^{-1}=k^{-1}\,\delta_{\beta}\,k.\]
    \noindent Now, using the fact that the action of $\W[\Gamma]$ on $\KVA[\Gamma]$ is given by $,w\,\delta_{\beta}\,w^{-1}=\delta_{w(\beta)}$, we obtain that 
    \[k^{-1}\,\delta_{\beta}\,k\,=\,\delta_{w(\beta)}.\]
    \noindent In particular, by the action of $\W[\Gamma]$ on its root system, for all $\beta$ in $\Phi[\Gamma]$ and for all $w\in \W[\Gamma]$, $w(\beta)$ belongs to $\Phi[\Gamma]$. Thus, $k\in \KVA[\Gamma]=\A[\hGamma]$ is such that $k^{-1}\widehat{\mathcal{S}}\,k=\widehat{\mathcal{S}}$, which means that $k$ belongs to $QZ(\A[\hGamma])$. Let $\omega(k)\in \W[\hGamma]$ be the image of $k$ under $\omega$. Since $\omega$ respects the product, we must have that \[\omega(k)^{-1}\;\widehat{S}\;\omega(k)=\omega(k^{-1}\,\widehat{\mathcal{S}}\,k)=\omega(\widehat{\mathcal{S}})=\widehat{S}.\]
    In other words, $\omega(k)^{-1}\in QZ(\W[\hGamma])$, which is trivial by Corollary \ref{QZWhGamma}. Then, $\omega(k)=\id$, and for all $\beta\in \Phi[\Gamma]$, we have \[\omega(\delta_{\beta})=\omega(k^{-1}\,\delta_{\beta}\,k)=\omega(\delta_{w(\beta)})\quad \Longleftrightarrow \quad \delta_{\beta}=\delta_{w(\beta)}\quad \Longleftrightarrow \quad \beta=w(\beta).\]
    \noindent But the only element $w\in \W[\Gamma]$ that fixes each root in $\Phi[\Gamma]$ is $w=\id$. Then, $g\in Z_{\VA[\Gamma]}(\KVA[\Gamma])$ is $g=k\in \KVA[\Gamma]$, so $g\in Z(\KVA[\Gamma])$, but by \cite{BellParThiel} we know that $Z(\KVA[\Gamma])=\{\id\}$. This shows that $g=\id$ and thus $Z_{\VA[\Gamma]}(\KVA[\Gamma])=\{\id\}$.
\end{proof}

\section{Indecomposability of $\VA[\Gamma]$}\label{indecva}
In this section, we show that for any connected Coxeter graph $\Gamma$, the virtual Artin group $\VA[\Gamma]$ is indecomposable.\medskip
\\
\noindent To establish this result, we distinguish between two cases based on the center of the associated Coxeter group: whether $Z(\W[\Gamma])=\{\id\}$ or $Z(\W[\Gamma])\neq\{\id\}$. If $Z(\W[\Gamma])$ is trivial, we can quickly deduce the indecomposability of $\VA[\Gamma]$. This case is addressed in Subsection \ref{trivialcentercase}.\medskip \\
\noindent Recall that a \textit{Remak decomposition of a group $G$} is a decomposition of $G$ as a direct product of finitely
many non-trivial indecomposable subgroups.  An automorphism $\varphi: G \longrightarrow G$ is said to be \textit{central}
if it is the identity modulo the center of $G$, that is, if $g^{-1}\varphi(g) \in Z(G)$ for all $g\in G$. By a result known as the Krull-Remak-Schmidt theorem, any finite group has a decomposition
as a direct product of indecomposable groups, which is unique up to a central automorphism
and a permutation of the factors.

\begin{reptheorem}{decompW}
\cite[Thm 4.1]{Par07}
Let $\Gamma$ be a connected spherical type Coxeter graph, and let $(\W[\Gamma],S)$ be its Coxeter system. Then $\W[\Gamma]$ is decomposable if and only if $\Gamma\in \{I_2(p)\,|\, p\geq 6 \; \mbox{and} \; p\equiv 2\, (\mathrm{mod} \,4)\}\cup \{B_n\,|\, n\geq 3\; \mbox{and}\; n \;\mbox{odd}\}\cup \{E_7,H_3\}$. In these cases, a Remak decomposition of $\W[\Gamma]$ is isomorphic to
\[
\W[\Gamma]=Z(\W[\Gamma])\,\times\, {\W[\Gamma]}/{Z(\W[\Gamma])}.
\]
\end{reptheorem}

\noindent This theorem implies that whenever a spherical-type Coxeter group $\W[\Gamma]$ is decomposable, its center $Z(\W[\Gamma])$ always appears as a direct indecomposable factor. If we assume that $Z(\W[\Gamma])=\{\id\}$, we immediately obtain the following result:
\begin{cor}\label{trivialcenterthenindecompo}
    Let $\Gamma$ be a connected Coxeter graph, and let $\W[\Gamma]$ be its associated Coxeter group. If $Z(\W[\Gamma])=\{\id\}$, then $\W[\Gamma]$ is indecomposable.
\end{cor}
\begin{proof}
It is known that the center of a Coxeter group $\W[\Gamma]$ is non-trivial if and only if $\Gamma$ is of spherical type and the longest element $w_0$ is central. Thus, if $Z(\W[\Gamma])=\{\id\}$, then one of the following holds:
\begin{itemize}
    \item $\Gamma$ is infinite.
    \item $\Gamma$ is finite but not of spherical type.
    \item $\Gamma$ is finite and of spherical type, but its longest element $w_0$ is not central.
\end{itemize}

\noindent In the first two cases, it follows from Theorem \ref{Wnonsphindec} and Theorem \ref{decompW} respectively that $\W[\Gamma]$ is indecomposable. In the third case, Theorem \ref{decompW} ensures that $\W[\Gamma]$ remains indecomposable. Therefore, under the assumption that $\W[\Gamma]$ has a trivial center, we conclude that $\W[\Gamma]$ is indecomposable.
\end{proof}
\noindent  We will use this result, along with the fact that the centralizer of $\KVA[\Gamma]$ in $\VA[\Gamma]$ is trivial, to deduce the indecomposability of the virtual Artin group in Lemma \ref{vaindectrivcent}.\\
\\
\noindent If $Z(\W[\Gamma])\neq\{\id\}$, then $\W[\Gamma]$ must be of spherical type, with its longest element $w_0$ lying in the center. In Theorem \ref{decompW}, Paris demonstrates that in this scenario, $\W[\Gamma]$ may admit a decomposition.
For instance, consider the well-known case of the dihedral group $\W[I_2(6)]$, which is isomorphic to the direct product $\mathfrak{S}_3\times \mathbb{Z}_2=\W[A_2]\times \W[A_1]$. Paris classifies the cases in which this decomposition occurs.\medskip \\
\noindent
We analyze the case $Z(\W[\Gamma])=\{\id,w_0\}$ in Subsection \ref{nontrivialcentercase}, showing that even though $\W[\Gamma]$ can be decomposable, the associated virtual Artin group $\VA[\Gamma]$ is always indecomposable.

\subsection{Trivial center case}\label{trivialcentercase}

In this subsection, we use the indecomposability of $\W[\Gamma]$ under the assumption that $Z(\W[\Gamma])=\{\id\}$ to prove that $\VA[\Gamma]$ is also indecomposable. Specifically, we assume for contradiction that $\VA[\Gamma]$ admits a direct decomposition as $H\times K$. We then project these factors onto $\W[\Gamma]$ through $\pi_K$ and analyze their images. Finally, we show that one of the factors, either $H$ or $K$, must be trivial, thereby concluding that $\VA[\Gamma]$ cannot be decomposable.

\noindent 
\begin{lem}\label{vaindectrivcent}
    Let $\Gamma$ be a Coxeter graph such that $Z(\W[\Gamma])=\{\id\}$. Then $\VA[\Gamma]$ is indecomposable.
\end{lem}
\begin{proof}
    Suppose that there exist subgroups $H,K<\VA[\Gamma]$ such that $\VA[\Gamma]=H\times K$. We will show that one of the two subgroups must be trivial.\medskip \\
     \noindent Consider the projection $\pi_K:\VA[\Gamma]=H\times K\longrightarrow \W[\Gamma]$ defined in Subsection \ref{subs-virtart}, and in particular consider the images $\pi_K(H)$ and $\pi_K(K)$. Necessarily, $\pi_K(H)\cdot\pi_K(K)=\W[\Gamma]$, and the elements in their intersection must commute with the entire group. Thus $\pi_K(H)\cap\pi_K(K)\subset Z(\W[\Gamma])$, that is trivial by hypothesis.
    Therefore $\pi_K(H)\times\pi_K(K)$ is a decomposition of $\W[\Gamma]$, which is indecomposable by Corollary \ref{trivialcenterthenindecompo}. Hence, one of the factors must be trivial. Suppose that $\pi_K(K)=\{\id\}$, which implies that $K<\ker(\pi_K)=\KVA[\Gamma]$. \medskip
    \\
    \noindent Assuming that $\VA[\Gamma]$ decomposes as $H\times K$, we also deduce that $\KVA[\Gamma]$ decomposes as \[(\KVA[\Gamma]\cap K)\times (\KVA[\Gamma]\cap H).\]
    Indeed, take $g\in \KVA[\Gamma]<\VA[\Gamma]=H\times K$. Then, $g$ can be written as $g=h\,\cdot k $ with $k\in K<\KVA[\Gamma]$ and $h\in H$. Hence $h=gk^{-1}\in \KVA[\Gamma]\cap H $. By Theorem \ref{kvaindecnonspherical}, $\KVA[\Gamma]$ is indecomposable, which implies that $\KVA[\Gamma]\cap K=\{\id\}$ or $H\cap\KVA[\Gamma]=\{\id\}$. Since we assumed $K<\KVA[\Gamma]$, the equality $\KVA[\Gamma]\cap K=\{\id\}$ implies $K=\{\id\}$ and we are done. If $\KVA[\Gamma]\cap H=\{\id\}$, then $\KVA[\Gamma]\cap K=\KVA[\Gamma]$. Since we knew that $K<\KVA[\Gamma]$, we conclude that $\KVA[\Gamma]=K$.\medskip \\
    \noindent Now, $H$ is a normal subgroup of $\VA[\Gamma]$ such that $\VA[\Gamma]=\KVA[\Gamma]\times H$, and in particular, $H$ commutes with $\KVA[\Gamma]$. Then, $H<Z_{\VA[\Gamma]}(\KVA[\Gamma])$, which is trivial by Theorem \ref{centralizerKVA}. Finally, we have shown that $H=\{\id\}$, so $\VA[\Gamma]$ is indecomposable. 
\end{proof}

\subsection{Non-trivial center case}\label{nontrivialcentercase}
In this subsection we prove that $\VA[\Gamma]$ is indecomposable when $\Gamma$ is of spherical type and the Coxeter group $\W[\Gamma]$ has non-trivial center. Together with the results of Subsection \ref{trivialcentercase}, this establishes the result of indecomposability of the virtual Artin group for any connected Coxeter graph $\Gamma$. \medskip
\\
\noindent It is important to note that in the proof of Lemma \ref{vaindectrivcent}, the assumption that $Z(\W[\Gamma])$ is trivial plays a crucial role. Without this assumption, we can no longer assert that $\pi_K(H)\cap\pi_K(K)$ is trivial. Moreover, the indecomposability of $\W[\Gamma]$ cannot be used to deduce information about $H$ and $K$, since in this case, $\W[\Gamma]$ may be decomposable with its center as a direct factor. These obstacles prevent us from applying the same argument as in the proof of Lemma \ref{vaindectrivcent}. \medskip \\
\noindent
To show that $\VA[\Gamma]$ remains indecomposable even when $Z(\W[\Gamma])\neq\{\id\}$, we proceed in steps. As before, we assume that $\VA[\Gamma]=H\times K$ and we perform the projection $\pi_K:H\times K\longrightarrow\W[\Gamma]$. Since $\pi_K$ is surjective, we obtain $\pi_K(K)\cdot\pi_K(H)=\W[\Gamma]$.
Furthermore, their intersection $\pi_K(K)\cap\pi_K(H)$ must commute with the entirety of $\W[\Gamma]$, so it must be contained in its center $Z(\W[\Gamma])=\{\id,w_0\}$. \medskip \\
Since $Z(\W[\Gamma])$ is cyclic of order two generated by the (central) longest element $w_0$, this leads to two possible cases:
\begin{enumerate}
    \item [\textbf{(1.)}]$\pi_K(K)\cap\pi_K(H)=\{\id\}$;
    \item [\textbf{(2.)}]$\pi_K(K)\cap\pi_K(H)=Z(\W[\Gamma])$.
\end{enumerate}
\noindent In the remainder of this section we will first show that case \textbf{(2.)} leads to a contradiction. Once this is established, it follows that the intersection $\pi_K(H)\cap\pi_K(H)$ must necessarily be trivial, as in case \textbf{(1.)}. We then show in Lemma \ref{vaindecnontrivcenter} that one of the two factors $H$ or $K$ must be trivial, thereby proving that $\VA[\Gamma]$ is indecomposable when $\W[\Gamma]$ has non-trivial center.

\begin{lem}\label{case2impossibile}
Let $\Gamma$ be a Coxeter graph such that $Z(\W[\Gamma])\neq \{\id\}$. Suppose that $\VA[\Gamma]$ decomposes as $H\times K$. Then $\pi_K(K)\,\cap\,\pi_K(H)=\{\id\}$.
\end{lem}
\noindent To show Lemma \ref{case2impossibile}, we first establish a preliminary result about the action of the (central) longest element $w_0\in \W[\Gamma]$ on the kernel $\KVA[\Gamma]$.\medskip
\\
\noindent Recall that the virtual Artin group $\VA[\Gamma]$ has the semidirect product structure $\VA[\Gamma]=\KVA[\Gamma]\rtimes \W[\Gamma]$, where $\W[\Gamma]$ acts on $\KVA[\Gamma]$ by conjugation as follows: \[
    w (k)= \iota_{\W}(w) \;k\;\iota_{\W}(w)^{-1},
    \]
    for $w\in \W[\Gamma]$ and $k\in \KVA[\Gamma]$. Sometimes, with abuse of notation, we write $w$ its image $\iota_W(w)\in \VA[\Gamma]$, as explained in Notation \ref{notationactionWonKVA}.

\begin{lem}\label{w0dnonfissak}
    Let $\Gamma$ be a Coxeter graph of spherical type such that $Z(\W[\Gamma])=\{\id,w_0\}$, where $w_0$ is the central longest element of $\W[\Gamma]$. Consider the action of the Coxeter group $\W[\Gamma]$ on the kernel $\KVA[\Gamma]$, given by the semidirect product structure $\VA[\Gamma]=\KVA[\Gamma]\rtimes \W[\Gamma]$, as described in Notation \ref{notationactionWonKVA}. Let $k$ be an element in $\KVA[\Gamma]$. Then, $w_0(k)=k$ implies that $k=\id$. Namely, the action of the central longest element cannot fix any non-trivial $k$ in $\KVA[\Gamma]$.
\end{lem}
\begin{proof}
    Recall that $\KVA[\Gamma]=\A[\hGamma]$, where the set of vertices $V(\hGamma)$ is the root system $\Phi[\Gamma]$, that is finite in this case. Pick any $\beta\in \Phi[\Gamma]$, and consider its opposite root $-\beta\in \Phi[\Gamma]$. By definition of $\hGamma$, we know that $\hm_{\beta,-\beta}=\infty$. Define:
    \begin{align*}
    &\cX^+_{\beta}=\Phi[\Gamma]\backslash\{-\beta\}, \qquad &\cX^-_{\beta}=\Phi[\Gamma]\backslash\{\beta\}, \qquad&\quad \quad\cY_{\beta}=\Phi[\Gamma]\backslash\{-\beta,\beta\}=\cX^+_{\beta}\cap \cX^-_{\beta}.
    \end{align*}
    Using Lemma \ref{Artingroupsareamalgamated}, we find:
    \[ \A[\hGamma]=\A[\hGamma_{\cX^+_{\beta}}]\quad\frpp _{\A[\hGamma_{\cY_{\beta}]}} \;\A[\hGamma_{\cX^-_{\beta}}] .\]
    We also know that the longest element $w_0$, when it is central, sends each root to its opposite, i.e., $w_0(\beta)=-\beta$ for all $\beta\in \Phi[\Gamma]$. Thus, the action of $w_0$ is an automorphism $w_0:\A[\hGamma]\longrightarrow \A[\hGamma]$ of order two, satisfying: 
\[w_0(\A[\hGamma_{\cX^+_{\beta}}])=\A[\hGamma_{\cX^-_{\beta}}]; \qquad \qquad w_0(\A[\hGamma_{\cX^-_{\beta}}])=\A[\hGamma_{\cX^+_{\beta}}].\]
By Lemma 3.6 in \cite{BellPar20}, the set $\{k\in \A[\Gamma]\;|\;w_0(k)=k\}$ is a subset of $\A[\hGamma_{\cY_{\beta}}]$. But this must hold for every root $\beta\in \Phi[\Gamma]$, we find
    \[
    k\in \bigcap_{\beta\in\Phi[\Gamma]}\A[\hGamma_{\cY_{\beta}}]=\A[\hGamma_{\cap_{\beta\in \Phi[\Gamma]}\cY_{\beta}}]=\A[\hGamma_{\emptyset}]=\{\id\}.
    \]
   Here, the first equality follows from Van der Lek's result in his Ph.D. thesis \cite{van1983homotopy}, stating that the intersection of standard parabolic subgroups of an Artin group is the Artin group on the intersection of their generating sets. By the definition of $\mathcal{Y}_{\beta}$, the intersection over all $\beta\in \Phi[\Gamma]$ is empty. Thus, if $k\in\KVA[\Gamma]$ satisfies $w_0(k)=k$, we conclude $k=\id$.
\end{proof}

\begin{proof}[Proof of Lemma \ref{case2impossibile}] 
    Consider the centralizer of the Coxeter group $\W[\Gamma]$ in $\VA[\Gamma]=\KVA[\Gamma]\rtimes \W[\Gamma]$. Thanks to \cite[Theorem 3.3]{BellParThiel}, we know that it is equal to the center of the Coxeter group $Z(\W[\Gamma])$, which is non-trivial. If $\VA[\Gamma]$ decomposes as $H\times K$, then, by Lemma \ref{lemmadecompcentralizer} we can write:
    \[ \{1,w_0\}=Z(\W[\Gamma])=Z_{\VA[\Gamma]}(\W[\Gamma])=(Z_{\VA[\Gamma]}(\W[\Gamma])\cap H)\times (Z_{\VA[\Gamma]}(\W[\Gamma])\cap K).\]
    \noindent Since the cyclic subgroup of order two is indecomposable, we have that one of the two factors must be trivial. Therefore, either $Z(\W[\Gamma])\cap K =\{\id\}$, or $Z(\W[\Gamma])\cap H=\{\id\}$. \medskip\\
    \noindent Suppose first that $Z(\W[\Gamma])\cap H =\{\id\}$ and $Z(\W[\Gamma])<K$. This implies that the longest element $w_0$ belongs to $K$. The group $H$ must commute with all the elements in $K$, and in particular it must commute with $w_0$. Take any $h\in H<\VA[\Gamma]=\KVA[\Gamma]\rtimes \W[\Gamma]$, and write it as $h=k_h\,w$, with $k_h\in \KVA[\Gamma]$ and $w\in \W[\Gamma]$. Hence:
    \[hw_0=w_0h\qquad \Longleftrightarrow \qquad k_hww_0=w_0k_hw\qquad \Longleftrightarrow \qquad k_hw_0\cancel{w}=w_0k_h\cancel{w}\qquad \Longleftrightarrow\qquad w_0(k_h)=k_h.\]
    \noindent By Lemma \ref{w0dnonfissak}, $w_0$ cannot fix any non-trivial element in $\KVA[\Gamma]$, thus $k_h=\id$ and $h=w\in \W[\Gamma]$. This means that $H<\W[\Gamma]$, but we also knew that $H\cap Z(\W[\Gamma])=\{\id\}$.\medskip\\
   \noindent By hypothesis, we know that $\pi_K(H)\cap \pi_K(K)$ contains $w_0$. Therefore, there exists some element $h\in H$ such that $\pi_K(h)=w_0$, but since $h=w\in \W[\Gamma]$, we have $\pi_K(h)=w=w_0$, which is absurd because we had supposed that $H\cap Z(\W[\Gamma])=\{\id\}$. Then, we have $Z(\W[\Gamma])\cap H\neq \{\id\}$. \medskip \\
    \noindent If $Z(\W[\Gamma])\cap K =\{\id\}$ and $Z(\W[\Gamma])< H$, we find analogously a contradiction. Therefore, the intersection $\pi_K(H)\cap \pi_K(K)$ is trivial. 
\end{proof}

\noindent Now that we have established that $\pi_K(H)\cap \pi_K(K)=\{\id\}$, we proceed to show that one of the two direct factors in the decomposition of $\VA[\Gamma]$ as $H\times K$ must be trivial.\medskip \\
\noindent If one between $\pi_K(H)$ and $\pi_K(K)$ is trivial - for example, if $\pi_K(K)=\{\id\}$ and $\pi_K(H)=\W[\Gamma]$ - then we can use the same argument as in the proof of Theorem \ref{vaindectrivcent}. Namely, we first show that $\KVA[\Gamma]$ decomposes as $(\KVA[\Gamma]\cap K)\times (\KVA[\Gamma]\cap H)$, then that $\VA[\Gamma]$ decomposes as $H\times \KVA[\Gamma]$, and finally, we conclude that $H=\{\id\}$ thanks to the result on the trivial centralizer of $\KVA[\Gamma]$.  \medskip\\
\noindent Thus, we may suppose that both $\pi_K(H)$ and $\pi_K(K)$ are non-trivial.\medskip \\
\noindent In the following lemma, we show that if $\W[\Gamma]$ has a nontrivial center, then $\VA[\Gamma]$ is indecomposable.

\begin{lem}\label{vaindecnontrivcenter}
    Let $\Gamma$ be a Coxeter graph such that $Z(\W[\Gamma])\neq \{\id\}$. Then $\VA[\Gamma]$ is indecomposable.
\end{lem}

\begin{proof}
Suppose that $\VA[\Gamma]=H\times K$. We aim to show that either $H=\{\id\}$ or $K=\{\id\}$.
    By Lemma \ref{case2impossibile}, we know that $\pi_K(K)\,\cap\,\pi_K(H)=\{\id\}$.\\
    If $\VA[\Gamma]$ allows the decomposition $H\times K$, then $\KVA[\Gamma]$ decomposes as \[\KVA[\Gamma]=(\KVA[\Gamma]\cap H)\,\times\, (\KVA[\Gamma]\cap K).\]
    This follows because any element $g\in \KVA[\Gamma]<\VA[\Gamma]$ can be written as $g=h\cdot k$ with $h\in H$, $k\in K$. By definition, $\pi_K(g)=\id=\pi_K(h)\cdot\pi_K(k)$, which implies that $\pi_K(h)=\pi_K(k^{-1})$. Since $\pi_K(K)\cap\pi_K(H)=\{\id\}$, we conclude that $\pi_K(h)=\pi_K(k)=\id$, so $h,k\in \KVA[\Gamma]$. Therefore, $\KVA[\Gamma]$ decomposes as $(\KVA[\Gamma]\cap H)\times (\KVA[\Gamma]\cap K)$.\medskip\\
    By Corollary \ref{kvaindecspherical}, we know that $\KVA[\Gamma]$ is indecomposable. Hence, either $(\KVA[\Gamma]\cap K)=\KVA[\Gamma]$ and $(\KVA[\Gamma]\cap H)=\{\id\}$, or $(\KVA[\Gamma]\cap H)=\KVA[\Gamma]$ and $(\KVA[\Gamma]\cap K)=\{\id\}$.\medskip\\
    We analyze the first case, as the second is analogous. If $(\KVA[\Gamma]\cap K)=\KVA[\Gamma]$, we have $\KVA[\Gamma]<K$. Since $H\subset Z_{\VA[\Gamma]}(K)$, it follows that that $H\subset Z_{\VA[\Gamma]}(\KVA[\Gamma])$. However, by Theorem \ref{centralizerKVA}, the centralizer $Z_{\VA[\Gamma]}(\KVA[\Gamma])$ is trivial. Therefore, $H=\{\id\}$. \medskip \\
    \noindent If we suppose that $(\KVA[\Gamma]\cap H)=\KVA[\Gamma]$ and $(\KVA[\Gamma]\cap K)=\{\id\}$, we get that $\KVA[\Gamma]<H$ and we conclude that $K=\{\id\}$ in the same way. Hence, if we assume $\VA[\Gamma]$ decomposable as $H\times K$, we obtain that one of the two factors is trivial.
\end{proof}

\noindent Finally, the union of Lemmas \ref{vaindectrivcent} and \ref{vaindecnontrivcenter} provides the following theorem:

\begin{thm}\label{vaindec}
    Let $\Gamma$ be a connected Coxeter graph. Then $\VA[\Gamma]$ is indecomposable.
\end{thm}
\noindent This final result completely answers the question of decomposability for virtual Artin groups. 

\section{Automorphism group}\label{autgroup}
In the previous section, we demonstrated that any irreducible virtual Artin group  $\VA[\Gamma]$ is indecomposable. It is worth noting that, in the study of Coxeter and Artin groups, it is quite rare to obtain results that hold in full generality without imposing assumptions on the Coxeter graph $\Gamma$.\medskip \\
\noindent In this section, we extend this result to derive several properties of the automorphism group of $\VA[\Gamma]$.\medskip \\
\noindent By definition, it is straightforward to observe that if $\Gamma$ is a Coxeter graph with connected components $\Gamma_1,\ldots,\Gamma_k$, then $\VA[\Gamma]$ decomposes as the direct product \[\VA[\Gamma]=\VA[\Gamma_1]\times \cdots\times \VA[\Gamma_k].\]
\noindent
Let $G$ be a group. We recall that a \textit{Remak decomposition} of $G$ is a decomposition of the form $G=H_1\times \cdots \times H_l$, where each $H_i$ is non-trivial and indecomposable, for all $i\in \{1,\ldots,l\}$. Using the result on indecomposability, we derive the following theorem.

\begin{thm}\label{RemakVA}
Let $\Gamma$ be a Coxeter graph, and let $\Gamma_1,\ldots , \Gamma_k$ be its connected components. Let $\VA[\Gamma]=H_1\times \cdots \times H_l$ be a Remak decomposition of the virtual Artin group associated with $\Gamma$. Then, $k=l$ and, up to permutation, $H_i=\VA[\Gamma_i]$ for all $i\in \{1,\ldots,k\}$.
\end{thm}
\begin{proof}
   We know that each $\Gamma_i$ is connected. Therefore, from the previous section, $\VA[\Gamma_i]$ is non-trivial and indecomposable for all $i\in \{1,\ldots,k\}$. Consequently, the decomposition $\VA[\Gamma]=\VA[\Gamma_1]\times \cdots \times \VA[\Gamma_k]$ forms a Remak decomposition of $\VA[\Gamma]$.\medskip
    \\
    \noindent As shown in \cite[Corollary 3.4]{BellParThiel}, the center $Z(\VA[\Gamma])$ is trivial for any $\Gamma$ Coxeter graph. By hypothesis, we can also decompose $\VA[\Gamma]$ as $H_1\times \cdots\times H_l$, where each $H_j$ is non-trivial and indecomposable, for $j\in \{1,\ldots,l\}$. Notably, $Z(\VA[\Gamma])=Z(H_1)\times \cdots\times Z(H_l)$, implying that each factor must have trivial center. Furthermore, the centralizer of a factor $Z_{\VA[\Gamma]}(H_1\times \cdots \times H_{j-1}\times H_{j+1}\times \cdots \times H_l)$ is given by $H_j$, for $j\in \{1,\ldots,l\}$. \medskip \\
    \noindent
    We now have the following two Remak decompositions:
    \begin{equation}
    \label{eqremakdec} \VA[\Gamma]=H_1\times \cdots \times H_l=\VA[\Gamma_1]\times \cdots \times\VA[\Gamma_k].\end{equation}
    \noindent We aim to show that, up to a permutation of the factors, these two decompositions coincide.\\
    Thanks to Lemma \ref{lemmadecompcentralizer}, we know that the centralizer $Z_{\VA[\Gamma]}(H)$ of a subgroup $H\leq \VA[\Gamma]$ decomposes as \[(Z_{\VA[\Gamma]}(H)\,\cap\,\VA[\Gamma_1])\times \cdots \times (Z_{\VA[\Gamma]}(H)\,\cap\,\VA[\Gamma_k]).\]
    \noindent In particular, since $Z_{\VA[\Gamma]}(H_2\times \cdots \times H_l)=H_1$, we deduce that $H_1= (H_1\,\cap\,\VA[\Gamma_1])\times \cdots \times (H_1\,\cap\, \VA[\Gamma_k])$. Since the factor $H_1$ is indecomposable, it follows that $H_1\,\cap\, \VA[\Gamma_i]=H_1$ for some $i\in \{1,\ldots,k\}$, and $H_1\,\cap\,\VA[\Gamma_h]=\{\id\}$ for $h\neq i$. Without loss of generality, assume $i=1$, and so that $\,H_1\leq\VA[\Gamma_1]$. \medskip \\
    \noindent Using the same reasoning, we show that the centralizer $Z_{\VA[\Gamma]}(\VA[\Gamma_2]\times \cdots \times \VA[\Gamma_k])=\VA[\Gamma_1]$ decomposes as $(\VA[\Gamma_1]\,\cap\,H_1)\times \cdots \times (\VA[\Gamma_1]\,\cap\,H_l)$. Since $\VA[\Gamma_1]$ is indecomposable by Theorem \ref{vaindec}, only one factor is non-trivial. By the above considerations, this must be $\VA[\Gamma_1]\,\cap\, H_1=\VA[\Gamma_1]$. Thus, $\VA[\Gamma_1]\leq H_1$, and with the reverse inclusion we conclude $\VA[\Gamma_1]=H_1$.\medskip\\
    \noindent Now we move to study the centralizer of $H_1\times H_3\times \cdots \times H_l$ in $\VA[\Gamma]$. Again, $Z_{\VA[\Gamma]}(H_1\times H_3\times \cdots \times H_l)=H_2$, and applying Lemma \ref{lemmadecompcentralizer} we get that $H_2$ decomposes as $H_2= (H_2\,\cap\,\VA[\Gamma_1])\times \cdots \times (H_2\,\cap\, \VA[\Gamma_k])$. However, the indecomposability of $H_2$ implies that only one factor in the previous decomposition is non-trivial, and thus equal to $H_2$. Since we have shown that $\VA[\Gamma_1]=H_1$, we must have
    \[
    H_2= (H_2\,\cap\,H_1)\times(H_2\cap\VA[\Gamma_2])\times \cdots \times (H_2\,\cap\, \VA[\Gamma_k]).
    \]
    We know that $H_1\cap H_2=\{\id\}$, hence 
    \begin{equation}
        H_2= (H_2\cap\,\VA[\Gamma_2])\times \cdots \times (H_2\,\cap\, \VA[\Gamma_k]).
    \end{equation}
\noindent As we did before, thanks to the indecomposability of $H_2$ we obtain that there exists some $i\in \{2,\ldots,k\}$ such that $H_2\cap\,\VA[\Gamma_i]=H_2$ and $H_2\cap\,\VA[\Gamma_h]=\{\id\}$ for all $h\neq i$ with $h\in \{2,\ldots,k\}$. Modulo a permutation of the factors, we can suppose that $i=2$ and that $H_2\cap\VA[\Gamma_2]=H_2$, thus that $H_2<\VA[\Gamma_2]$.\medskip \\
\noindent By repeating the same argument as before, we show that $\VA[\Gamma_2]<H_2$ and therefore that $\VA[\Gamma_2]=H_2$.
 By recursively applying this reasoning, we get that each $H_j$ for $j\in \{1,\ldots,l\}$ must contain and be contained in $\VA[\Gamma_i]$, with $i$ in $\{j,\ldots,k\}$. Therefore, we establish that $l=k$ and, up to a permutation of the factors, $\VA[\Gamma_i]=H_i$, for all $i\in \{1,\ldots,k\}$.
\end{proof}

\noindent This theorem establishes that the only possible Remak decomposition of a virtual Artin group $\VA[\Gamma]$ is the one relative to its irreducible components. This result has a direct impact on understanding the automorphism group of $\VA[\Gamma]$. 

\begin{thm}\label{homoautvasn}
 Let $\Gamma$ be a Coxeter graph, and let $\Gamma_1,\ldots,\Gamma_k$ be its connected components. Then there exists an homomorphism $\Psi:\Aut(\VA[\Gamma])\longrightarrow \mathfrak{S}_k$ such that, for every $\varphi\in \Aut(\VA[\Gamma])$ and every $i\in \{1,\ldots,k\}$, 
 \[
\varphi(\VA[\Gamma_i])= \VA[\Gamma_{\Psi(\varphi)(i)}].
 \]
   
\end{thm}

\noindent The proof of Theorem \ref{homoautvasn} follows directly from Theorem \ref{RemakVA}. Observe now that the kernel of such homomorphism $\Psi:\Aut(\VA[\Gamma])\longrightarrow \mathfrak{S}_k$ is precisely the subgroup of automorphisms $\varphi\in \Aut(\VA[\Gamma])$ such that $\varphi(\VA[\Gamma_i])=\VA[\Gamma_i]$ for all $i\in \{1,\ldots,k\}$. In other words:
\[
\ker(\Psi)=\Aut(\VA[\Gamma_1])\times \cdots \times \Aut(\VA[\Gamma_k]).
\]
\begin{cor}\label{autfiniteindex}
    Let $\Gamma$ be a Coxeter graph, and let $\Gamma_1,\ldots,\Gamma_k$ be its connected components. Then \\$\Aut(\VA[\Gamma_1])\times \cdots \times \Aut(\VA[\Gamma_k])$ is a finite index normal subgroup of $\Aut(\VA[\Gamma])$.
\end{cor}

\noindent Corollary \ref{autfiniteindex} implies that, to study the automorphism group of a virtual Artin group, it suffices to understand the automorphism groups of its irreducible components.\medskip \\
\noindent To date, except in the case where $\Gamma=A_{n-1}$ and $\VA[\Gamma]=\VB_n$ with $n\geq 5$, the automorphism group of $\VA[\Gamma]$ has not been investigated. For the virtual braid group on $n$ strands, the detailed study of Bellingeri and Paris (\cite{BellPar20}) provides a comprehensive description of all homomorphisms, automorphisms, and exterior automorphisms of $\VB_n$, when $n\geq 5$.\medskip \\
\noindent Corollary \ref{autfiniteindex} marks an important step toward analyzing the automorphism group of a virtual Artin group.

\bigskip
\bigskip

\noindent\textit{\textbf{Federica Gavazzi} \\  Université Bourgogne Europe, CNRS, IMB UMR 5584, F-21000 Dijon, France.} \par
\noindent\textit{E-mail address:} \texttt{\href{mailto:Federica.Gavazzi@u-bourgogne.fr}{Federica.Gavazzi@u-bourgogne.fr}}

\printbibliography

\end{document}